\documentclass[10pt,a4paper]{article}
\usepackage{appendix}
\usepackage[numbers,sort&compress]{natbib}
\usepackage[T1]{fontenc}
\usepackage[utf8]{inputenc}
\usepackage{authblk}
\usepackage{amsmath,amssymb,amsthm,esint,bm}
\usepackage{mathrsfs}
\usepackage{bookmark}
\usepackage{amsmath}
\usepackage{enumitem}
\allowdisplaybreaks[4]

\newtheorem{definition}{Definition}[section]
\newtheorem{theorem}[definition]{Theorem}
\newtheorem{lemma}[definition]{Lemma}
\newtheorem{proposition}[definition]{Proposition}
\newtheorem{corollary}[definition]{Corollary}
\theoremstyle{remark}
\newtheorem{remark}[definition]{Remark}
\numberwithin{equation}{section}
\newcommand{\abs}[1]{\lvert#1\rvert}
\newcommand{\Abs}[1]{\left\lvert#1\right\rvert}

\newcommand{\rn}{{\mathbb{R}^d}}

\allowdisplaybreaks[4]

\title{Gradient regularity for degenerate fully nonlinear free transmission problems with Hamiltonian terms}

\author[a]{Wentao Huo}
\affil[a]{School of Mathematical Sciences, Nankai University, Tianjin 300071, P.R. China}
\date{\today}

\usepackage{hyperref}
\begin{document}
\maketitle
\footnotetext[1]{E-mail: huowentaoouc@163.com (W. Huo). 
}

\begin{abstract}
	We  develop  the  regularity  theory  of  viscosity  solutions 
	to  degenerate fully nonlinear free transmission  problems with Hamiltonian terms. By framing the equation in the context of viscosity inequalities, we establish local H\"{o}lder regularity of the gradient. In addition, based on a new improved oscillation-type estimate combined with a localized analysis, we obtain sharp pointwise $C^{1,\alpha}$ regularity. 	



\end{abstract}


\section{Introduction}\label{section1}
Transmission problems model physical phenomena in which the behavior changes across some fixed interface, and have attracted 
considerable attention throughout the years, starting with the pioneering work of Picone \cite{Picone1954} in elasticity in the 1950s and subsequent works \cite{Lions1956,Stampacchia1956,Campanato1957,Schechter1960}. This class of problems has a wide range of applications in various areas such as electromagnetic processes, composite materials, vibrating folded membranes and climatology. For a comprehensive study of these problems, we refer the reader to  monograph \cite{Borsuk2010}. For other recent developments, see \cite{CaffarelliARMA2021,LiARMA2000,Soria-CarroAdvMath2023,NirenbergCPAM2003,Citti2012,KriventsovARMA2015,PimentelJDE2022,PimentelSantos2023,Teixeira2015CMP} and references therein.

In this paper, we are concerned with the following fully nonlinear free transmission problems with general variable exponents and Hamiltonian terms
\begin{equation}\label{11model}
\bigg(|Du|^{\gamma(x,u,Du)}+a(x)|Du|^{\sigma(x,u,Du)}\bigg)F(D^{2}u,x)+\mathcal{H}({Du}, x)=f(x) \quad  \text{in} \quad B_{1},
\end{equation}
Here $B_{1} \subset \mathbb{R}^d$ $(d\geq 2)$ and the source term $f\in C({B_{1}}) \cap L^{\infty}(B_{1})$. The fully nonlinear operator $F:S^{d}\times {B_{1}}\rightarrow \mathbb{R}$ is uniformly $(\lambda,\Lambda)$-elliptic with $F(0,\cdot) = 0$, that is, there exist constants $0<\lambda\leq \Lambda$ such that 
\begin{equation}\label{11}
\lambda\|N\|\leq F(M,x)-F(M+N,x)\leq \Lambda\|N\|
\end{equation}
for any $M,N\in S^{d}$ with $N\geq 0$ and $x\in {B_{1}}$. Moreover, we assume the uniform continuity on the coefficients of $F$, i.e., there exist constants $C_{F}>0$ and $\theta\in(0,1)$ such that
\begin{equation}\label{12}
	{\Omega}\ni 
	x,x_{0}\mapsto{\rm osc}_{F}(x,x_{0}):=\sup\limits_{M\in S^{d}\setminus \{0\}}\frac{\Abs{F(M,x)-F(M,x_{0})}}{\|M\|}\leq C_{F}|x-x_{0}|^{\theta},
\end{equation}
which measures the oscillation of coefficients of $F$ around $x_{0}$. For simplicity, we denote ${\rm osc}_{F}(x):={\rm osc}_{F}(x,0)$.
We shall require that the general variable exponents $\gamma,\sigma:B_{1}\times\mathbb{R}\times\rn$ are well-defined in its domain and fulfill
 \begin{equation}\label{13}
 	0<\gamma_{1}\leq \gamma(x,t,p)\leq \sigma(x,t,p)\leq \gamma_{2}<\infty
 \end{equation}
 for fixed constants $\gamma_{1}$ and $\gamma_{2}$, and the modulating function $a$ satisfies
 \begin{equation}\label{14}
 	0\leq a(\cdot)\in C({B_{1}}).
 \end{equation}
 The Hamiltonian term $\mathcal{H}:\mathbb{R}^{d} \times B_{1}\rightarrow \mathbb{R}$ is continuous and there exist constants $\mathcal{K},\mathcal{M}>0$ and $0<m\leq 1+\gamma_{1}$ such that
 \begin{equation}\label{15}
 	|\mathcal{H}(\xi,x)|\leq \mathcal{K}+\mathcal{M}|\xi|^{m}
 \end{equation}
 for every $\xi\in\mathbb{R}^{d}$, $x\in B_{1}$. Note that \eqref{11model} can be regarded as a free transmission problem, since the degeneracy law displayed in \eqref{11model} depends on the solution itself. In particular, when $a\equiv 0$ and $\gamma(x,u,Du)=\gamma_{1}\chi_{\{u>0\}}+\gamma_{2}\chi_{\{u<0\}}$, the degeneracy law develops discontinuities along $\partial\{u>0\}$ and $\partial\{u<0\}$. The various regions where each degeneracy regime is in force are in part unknown a priori as they vary according to the sign of solutions, the transmission interface can be interpreted as a free boundary and hence this class of models is referred to as free transmission problems. 

Observe that, when $\gamma(x,u,Du)=\gamma(x)$ and $\sigma(x,u,Du)=\sigma(x)$, \eqref{11model} is related to the typical degenerate fully nonlinear elliptic equation 
\begin{equation*}
	\bigg(|Du|^{\gamma(x)}+a(x)|Du|^{\sigma(x)}\bigg)F(D^{2}u,x)+\mathcal{H}({Du}, x)=f(x) \quad  \text{in} \quad B_{1}.
\end{equation*}
Several aspects of this class of partial differential equations has been widely investigated, such as comparison and maximum principles, well-posedness of the Dirichlet problem,  Liouville type theorems, Aleksandroff–Bakelman–Pucci estimates, Harnack inequalities, and regularity theory (cf. \cite{Birindelli2004,Birindelli2006,Birindelli2010JDE,Birindelli2015,Fang,Silva2020,Andrade,Andrade2022,B-Demengel2016,Baasandorj,Bronzi2020,Davil2010,Imbert1,Silva2023,Ricarte,Nascimento2025,Huo2026,B-Demengel2019,Junges2010,Birindelli2014ESAIM,Fili} and the references therein). 

Regarding free transposition problems, Huaroto et al. \cite{PimentelAnalPDE2023} considered the degenerate free transmission problem
\begin{equation*}
	|Du|^{\gamma_{1}\chi_{\{u>0\}}+\gamma_{2}\chi_{\{u<0\}}}F(D^{2}u)=f(x) \quad  \text{in} \quad \Omega
\end{equation*}
where $\Omega$ is a bounded domain in $\rn$ and $\gamma_{1},\gamma_{2}>0$. They proved the existence of a viscosity solution to the associated Dirichlet problem and established optimal regularity in $C^{1,\alpha}$-spaces, with appropriate estimates. Subsequently, De Filippis \cite{De Filippis2022} studied the following free transmission problems governed by fully nonlinear elliptic equations with nonhomogeneous degeneracies
\begin{equation*}
	\bigg(|Du|^{\gamma_{1}\chi_{\{u>0\}}+\gamma_{2}\chi_{\{u<0\}}}+a(x)\chi_{\{u>0\}}|Du|^{q}+b(x)\chi_{\{u<0\}}|Du|^{s}\bigg)F(D^{2}u)=f(x) \quad  \text{in} \quad \Omega
\end{equation*}
and proved the existence of solutions as well as obtained local optimal H\"{o}lder continuity for the gradient of viscosity solutions. In a recent paper, Jesus \cite{Jesus2022CVPDE} investigated the following free transmission problem with variable exponents
\begin{equation*}
	|Du|^{\gamma(x.u.Du)}F(D^{2}u)=f(x) \quad  \text{in} \quad B_{1},
\end{equation*}                                            
Under assumption \eqref{14}, the local $C^{1,\alpha}$ regularity of solutions was established. In addition, the author proved optimal pointwise regularity depending on the degeneracy rate under additional
suitable conditions on $\gamma(x.u.Du)$.
 
Despite the substantial progress achieved in the literature, there remains a gap in the optimal regularity theory for degenerate fully nonlinear free transmission  problems with Hamiltonian terms. To this end, inspired by the work mentioned above, we aim to establish the local H\"{o}lder regularity for the gradient and optimal pointwise $C^{1,\alpha}$ regularity of viscosity solutions to \eqref{11model}. The presence of Hamiltonian lower-order terms poses new difficulties and requires new arguments that will considerably differ from \cite{Jesus2022CVPDE,PimentelAnalPDE2023}, due to the interplay between the degeneracy and the Hamiltonian's growth. Our strategies can be applied to degenerate normalized $p$-Possion free transmission problems. 
 
To illustrate our results, we recall that, viscosity solutions to  homogeneous equation $F(D^{2}u)=0$ are of class $C_{\rm loc}^{1,\alpha_{0}}$ for a universal constant $\alpha_{0}=\alpha_{0}(d,\lambda,\Lambda)\in (0,1)$, see \cite[Chapter 5]{Caff1}. We are now in a position to state the local $C^{1,\alpha}$ regularity estimate of solutions to \eqref{11model}.
\begin{theorem}\label{thm1}
	Let $u\in C(B_{1})$ be a viscosity solution of \eqref{11model} under hypotheses \eqref{11}-\eqref{15}. Then $u\in C_{\rm loc}^{1,\alpha^{\prime}}(B_{1})$ for some $\alpha^{\prime}\in (0,\alpha_{0})\cap \left(0,\frac{1}{1+\gamma_{2}}\right]$ with the following estimate:
	\begin{itemize}
		\item [{\rm$({{\rm i}})$}] if $0<m<1+\gamma_{1}$, then
		\begin{equation*}
			\|u\|_{C^{1, \alpha^{\prime}}(B_{1/2})}\leq C\left(1+\|u\|_{L^{\infty}\left(B_{1}\right)}+\left(\|f\|_{L^{\infty}\left(B_{1}\right)}+\mathcal{K}\right)^{\frac{1}{1+\gamma_{1}}}+\mathcal{M}^{\frac{1}{1+\gamma_{1}-m}}\right),
		\end{equation*}
		where the constant $C$ depends on $d,\lambda,\Lambda,\alpha^{\prime},m,\gamma_{1}$;
		\item [{\rm$({{\rm ii}})$}] if $m=1+\gamma_{1}$, then
		\begin{equation*}
			\|u\|_{C^{1, \alpha^{\prime}}(B_{1/2})}\leq C\left(1+\|u\|_{L^{\infty}(B_{1})}\right),
		\end{equation*}
		where the constant $C$ depends in addition on $\|f\|_{L^{\infty}\left(B_{1}\right)}$, $\mathcal{K}$ and $\mathcal{M}$.
	\end{itemize}
\end{theorem}

It is worth pointing out that here we only require that $\gamma(x,u,Du)$ and $\sigma(x,u,Du)$ is bounded above and below. In addition, Theorem \ref{thm1} naturally 
extends the regularity results obtained in \cite{PimentelAnalPDE2023,Jesus2022CVPDE} to a more general setting. 

When the operator $F$ is concave, solutions to $F(D^{2}u)=0$ are locally of class $C^{1,1}$ by the Evans-Krylov theory \cite[Theorem 6.1]{Caff1}. In particular, taking $\alpha_{0}=1$ in Theorem \ref{thm1} yields the sharp regularity result, namely, viscosity solutions of \eqref{11model} are of class $C_{\rm loc}^{1,\frac{1}{1+\gamma_{2}}}$.

Observe that there is an intrinsic dependence between the obtained regularity and the rate of degeneracy. Therefore, if this rate is varied over the domain, it is natural to expect regularity results 
that vary over the domain as well. To this end, we first make some essential assumptions.  
Let \( \Omega_i(u,Du) \subset B_1 \), \( i = 1, 2, \ldots, N \), be disjoint sets which depend on the solution \( u \), and define  
\begin{equation*}
\Omega_0(u,Du) := B_1 \setminus \bigcup_{i=1}^N \Omega_i(u,Du).
\end{equation*}
Assume also \( \gamma(x,u,Du), \sigma(x,u,Du) \) have the form  
\begin{equation}\label{16}
	\gamma(x,u,Du):= \sum_{i=0}^N \gamma_i(x) \chi_{\Omega_i(u,Du)}, \quad
	\sigma(x,u,Du):= \sum_{i=0}^N \sigma_i(x) \chi_{\Omega_i(u,Du)} 
\end{equation}
with  
\begin{equation}\label{17}
	0 < \gamma_1 \leq \gamma_i(x) \leq \sigma_i(x) \leq \gamma_2 < +\infty,
\end{equation}
where \( \chi_{\Omega_i} \) are characteristic functions of \( \Omega_i \). Moreover, we require additional assumptions on \( \gamma_i(x) \) and \( \sigma_i(x) \), \( i = 0, 1, \cdots, N \). Suppose that there is a non-decreasing function \( \omega : [0, +\infty) \to [0, +\infty) \) vanishing at zero such that  
\begin{equation}\label{18}
	|\gamma_i(x) - \gamma_i(y)| + |\sigma_i(x) - \sigma_i(y)| \leq \omega(|x - y|), \quad i = 0, 1, \cdots, N,
\end{equation}
and $\omega$ satisfies 
\begin{equation}\label{19}
	\limsup_{t\rightarrow 0}\omega(t)\log(t^{-1})=0.
\end{equation}
\begin{remark}\label{rmk1}
	The condition \eqref{19} admits an equivalent assertion in \cite[Section 2]{Jesus2022CVPDE}: for any fixed  
	$\delta_1>0$ such that if $\varrho \leq \delta_1$, then for every $k \in \mathbb{N}$,
	\[
	k\,\omega(\varrho^k) \leq \frac{\alpha_0 - \alpha}{2},
	\]
	where $\alpha$ and $\alpha_0$ are the constants in Theorem \ref{thm2} below.	
\end{remark}	
\begin{remark}\label{rmk2}
Assumptions of the type \eqref{19} are standard when obtaining higher regularity of solutions to equations with variable exponents. For example, Acerbi and Mingione \cite{Mingione2001ARMA} proved the improved regularity to a class of variational problems with variable exponents under the assumption above.
\end{remark}
We proceed to give the second result concerning improved pointwise regularity.
\begin{theorem}\label{thm2}
	Let $u\in C(B_{1})$ be a viscosity solution of \eqref{11model} under hypotheses \eqref{11}-\eqref{12} and \eqref{14}-\eqref{19}. Then for every $x_{0}\in B_{1/2}$, $u$ is $C^{1,\alpha}$ at $x_{0}$ with  
	$$\alpha=\min_{i=0,1,\ldots, N}\left\{\alpha_{0}^{-},\frac{1}{\gamma_{i}(x_{0})}\right\}.$$
More precisely, there holds
\begin{itemize}
	\item [{\rm$({{\rm i}})$}] if $0<m<1+\gamma_{1}$, then
	\begin{equation*}
		[u]_{C^{1, \alpha}({x_0})}\leq C\left(1+\|u\|_{L^{\infty}\left(B_{1}\right)}+\left(\|f\|_{L^{\infty}\left(B_{1}\right)}+\mathcal{K}\right)^{\frac{1}{1+\gamma_{1}}}+\mathcal{M}^{\frac{1}{1+\gamma_{1}-m}}\right),
	\end{equation*}
	where the constant $C$ depends on $d,p,\alpha,m,\gamma_{1}$;
	\item [{\rm$({{\rm ii}})$}] if $m=1+\gamma_{1}$, then
	\begin{equation*}
		[u]_{C^{1, \alpha}({x_0})}\leq C\left(1+\|u\|_{L^{\infty}(B_{1})}\right),
	\end{equation*}
	where the constant $C$ depends in addition on $\|f\|_{L^{\infty}\left(B_{1}\right)}$, $\mathcal{K}$ and $\mathcal{M}$.
\end{itemize}
\end{theorem}

Theorem \ref{thm2} establishes optimal pointwise $C^{1,\alpha}$ regularity for problem \eqref{11model}, which gives an explicit characterization of $\alpha$ in terms of the degeneracy laws. The proof is based on a new improved oscillation-type estimate combined with a localized analysis. It is worth emphasizing that
our geometric approach is particularly refined and totally different from the strategy adopted in \cite{Jesus2022CVPDE}. Notably, even for the simplest model where $N=2$, $\Omega_0(u,Du)=\{u=0\}$, $\Omega_1(u,Du)=\{u>0\}$, $\Omega_2(u,Du)=\{u<0\}$ and $\gamma_{i}(x),\sigma_{i}(x)$ $(i=0,1,2)$ are constants, our finding is still new. 

The remainder of this paper is organized as follows. In Section \ref{section222}, we introduce the basic notions and some auxiliary results. In Section \ref{section333}, we present the proof of local $C^{1,\alpha^{\prime}}$ regularity estimate. Section \ref{section444} is dedicated to proving sharp pointwise H\"{o}lder regularity of the gradient stated in Theorem \ref{thm2}. 
\section{Preliminaries}\label{section222}
In this section, we first introduce the notations and some basic concepts, and subsequently present some useful auxiliary results, which are pivotal for proving our main theorems.

\subsection{Notations and basic concepts}
The following notations are used throughout the paper.
\begin{itemize}
	\item [$\bullet$] $S^{d}$ denotes the set of all real $d\times d$ symmetric matrices.
	\item [$\bullet$] For $r>0$, $B_{r}(x_{0})$ denotes the open ball with radius $r$ and centred at $x_{0}\in\rn$. In particular, we shall simply denote $B_{r}:=B_{r}(0)$.
	\item [$\bullet$] $I$ denotes the $d\times d$ identity matrix.
	\item [$\bullet$] $\langle \eta, \xi \rangle$ and $\eta\cdot \xi$ denote the inner product of vectors $\eta, \xi\in\rn$.
	\item [$\bullet$] $A\leq B$ in $S^{d}$ means that $\left\langle A\xi,\xi\right\rangle\leq \left\langle B\xi,\xi\right\rangle$ for any $\xi \in \mathbb{R}^{d}$. In other words, $B-A$ is positive semidefinite.
	\item [$\bullet$] $\chi_{E}$ denotes the characteristic function of measurable set $E$.
	\item [$\bullet$] For $a,b\in \rn$, we denote by $a\otimes b$ the $d\times d$-matrix for which $(a\otimes b)_{ij}=a_{i}b_{j}$.
	\item [$\bullet$] The symbol $C$ denotes a positive constant whose value may vary from line to line, and only the relevant dependencies are specified in parentheses. Besides, a constant is said to be universal if it depends at most upon the structure constants in \eqref{11}-\eqref{19}.
		\item [$\bullet$] We say $u\in C^{1,\alpha}(x_{0})$ if $u\in C^{1}$ in a neighborhood of $x_{0}$ and 
		$$[u]_{C^{1,\alpha}(x_{0})}:=\sup_{r>0,\;y\in B_{r}(x_{0})}\frac{|Du(y)-Du(x_{0})|}{|y-x|^{\alpha}}<\infty.$$
\end{itemize}

We proceed to introduce the definition of the Pucci extremal operators.
\begin{definition}
	Let $0<\lambda\leq \Lambda$ be as in \eqref{11}. The Pucci extremal operators $\mathcal{P}_{\lambda,\Lambda}^{\pm}:S^{d}\rightarrow \mathbb{R}$ are defined by
	$$\mathcal{P}_{\lambda,\Lambda}^{+}(M):=-\lambda\sum\limits_{e_{i}>0}e_{i}-\Lambda\sum\limits_{e_{i}<0}e_{i}\quad {\rm and}\quad \mathcal{P}_{\lambda,\Lambda}^{-}(M):=-\Lambda\sum\limits_{e_{i}>0}e_{i}-\lambda\sum\limits_{e_{i}<0}e_{i},$$
	where $\{e_{i}\}_{i=1}^{d}$ are the eigenvalues of the matrix $M$. 
\end{definition}

With the Pucci extremal operators in hand, the uniformly $(\lambda,\Lambda)$-ellipticity of the operator $F$ can be reformulated as
$$\mathcal{P}_{\lambda,\Lambda}^{-}(N)\leq F(M+N,x)-F(M,x)\leq \mathcal{P}_{\lambda,\Lambda}^{+}(N)$$
for all $M,N\in S^{d}$ and $x\in B_{1}$.

For the sake of completeness, we recall the notion of viscosity solution, see \cite{Crandle1,Caff1}. 
\begin{definition} \label{dingyi2}
Let \( G : S^{d}\times  \mathbb{R}^d \times \mathbb{R} \times B_{1}\) be a degenerate elliptic operator. We say that an upper semicontinuous function \( u : B_{1} \to \mathbb{R} \) is a viscosity subsolution to
	\begin{equation}\label{tuihuafangcheng}
			G(D^2u,Du,u,x) = 0  \quad {\rm in}\quad B_{1},
	\end{equation}
if, whenever \( \varphi \in C^2(B_{1}) \) and \( u - \varphi \) attains a local maximum at \( x_0 \in B_{1} \), we have
	\[
	G(D^2\varphi(x_0),D\varphi(x_0),u(x_0),x_0) \leq 0.
	\]
	Similarly, we say that a lower semicontinuous function \( u : B_{1} \to \mathbb{R} \) is a viscosity supersolution to \eqref{tuihuafangcheng} if, whenever \( \varphi \in C^2(B_{1}) \) and \( u - \varphi \) attains a local minimum at \( x_0 \in B_{1} \), we have
	\[
G(D^2\varphi(x_0),D\varphi(x_0),u(x_0),x_0) \geq 0.
	\]
	If \( u \) is both a viscosity subsolution and a viscosity supersolution to \eqref{tuihuafangcheng}, we say \( u \) is a viscosity solution to \eqref{tuihuafangcheng}.
\end{definition}
\subsection{ Auxiliary Results}
	\begin{lemma}\label{lem2.1}
	Let \( u \in C(B_1) \) be a viscosity solution to \eqref{11model} under assumptions \eqref{11}-\eqref{15}.  
	Then \( u \) is a viscosity subsolution to 
	\begin{equation}\label{aa11}
		\min_{i=1,2} \left\{ |Du|^{\gamma_i} F \left( D^2u,x\right)\right\} + \mathcal{H}({Du}, x)= \|f\|_{L^\infty(B_1)} \quad  \text{in} \quad B_{1}
	\end{equation}
	and a viscosity supersolution to 
	\begin{equation}\label{aa12}
		\max_{i=1,2} \left\{ |Du|^{\gamma_i} F \left( D^2u,x\right)\right\}+ \mathcal{H}({Du}, x)= -\|f\|_{L^\infty(B_1)} \quad  \text{in} \quad B_{1}.
	\end{equation}
\end{lemma}

\begin{proof}
	We only prove that if \( u \) is a viscosity subsolution to \eqref{11model}, then it is a subsolution to \eqref{aa11}, noting that the remaining case follows similarly.  
	Let \( \varphi \in C^2(B_1) \) and assume \( u - \varphi \) has a local maximum at \( x_0 \), then by the definition of viscosity subsolution, we have 
\begin{equation*}
		\begin{split}
		&\bigg(|D\varphi(x_0)|^{\gamma(x_0,u(x_0),D\varphi(x_0))}+a(x_{0})|D\varphi(x_0)|^{\sigma(x_0,u(x_0),D\varphi(x_0))}\bigg) F \left( D^2\varphi(x_0) ,x_{0}\right)\\ +&\mathcal{H}({D\varphi(x_{0})}, x_{0})\leq f(x_0).
		\end{split}
\end{equation*}
	{\bf Case 1.} $F \left( D^2\varphi(x_0),x_0 \right) \leq 0$. Then it is obvious that
	\[
	|D\varphi(x_0)|^{\gamma_{i}} F \left( D^2\varphi(x_0),x_{0} \right)+ \mathcal{H}({D\varphi(x_{0})}, x_{0})\leq \|f\|_{L^\infty(B_1)}, \quad i=1,2.
	\]
	{\bf Case 2.} $F \left( D^2\varphi(x_0),x_0 \right) \geq 0$. By direct calculations, one of the following inequalities must hold 
	\[
	|D\varphi(x_0)|^{\gamma_{1}} F \left( D^2\varphi(x_0),x_{0} \right)+ \mathcal{H}({D\varphi(x_{0})}, x_{0})\leq \|f\|_{L^\infty(B_1)},
	\]
	\[
	|D\varphi(x_0)|^{\gamma_{2}} F \left( D^2\varphi(x_0),x_{0} \right)+ \mathcal{H}({D\varphi(x_{0})}, x_{0})\leq \|f\|_{L^\infty(B_1)},
	\]
under either \( |D\varphi(x_0)| \geq 1 \) or \( |D\varphi(x_0)| < 1 \). 

Combining the above two cases, we conclude that  
	\begin{equation*}
		\min_{i=1,2} \left\{ |D\varphi(x_0)|^{\gamma_{i}} F \left( D^2\varphi(x_0),x_{0} \right)\right\}+ \mathcal{H}({D\varphi(x_{0})}, x_{0}) \leq  \|f\|_{L^\infty(B_1)}.
	\end{equation*}	
	Hence, we have proved that \( u \) is a subsolution of \eqref{aa11}.
\end{proof}

\begin{lemma}\label{lem2.2}
	Let \( u \in C(B_1) \) be a viscosity solution to \eqref{11model} under assumptions \eqref{11}-\eqref{12} and \eqref{14}-\eqref{19}.  
	Then \( u \) is a viscosity subsolution to 
	\begin{equation}\label{bb11}
		\min_{i=0,1,\ldots,N} \left\{ \bigg(|Du|^{\gamma_i(x)}+a(x)|Du|^{\sigma_i(x)} \bigg)F \left( D^2u,x\right)\right\} + \mathcal{H}({Du}, x)= \|f\|_{L^\infty(B_1)} \quad  \text{in} \quad B_{1}
	\end{equation}
	and a viscosity supersolution to 
	\begin{equation}\label{bb12}
		\max_{i=0,1,\ldots,N} \left\{ \bigg(|Du|^{\gamma_i(x)}+a(x)|Du|^{\sigma_i(x)} \bigg)F \left( D^2u,x\right)\right\} + \mathcal{H}({Du}, x)= -\|f\|_{L^\infty(B_1)} \quad  \text{in} \quad B_{1}.
	\end{equation}
\end{lemma}

\begin{proof}
	We first prove that \( u \) is a viscosity subsolution to \eqref{bb11}.  
	Let \( \varphi \in C^2(B_1) \) and assume \( u - \varphi \) has a local maximum at \( x_0 \), then we immediately obtain  
	\[
	\bigg(|D\varphi(x_0)|^{\gamma(x_0,u,Du)}+a(x_{0})|D\varphi(x_0)|^{\sigma(x_0,u,Du)}\bigg) F \left( D^2\varphi(x_0) ,x_{0}\right) +\mathcal{H}({D\varphi(x_{0})}, x_{0})\leq f(x_0),
	\]
	where
	\begin{equation*}
		\gamma(x,u,Du) = \sum_{i=0}^N \gamma_i(x) \chi_{\Omega_i(u,Du)}, \quad
		\sigma(x,u,Du)= \sum_{i=0}^N \sigma_i(x) \chi_{\Omega_i(u,Du)} 
	\end{equation*}
	By Theorem \ref{thm1} we know that \( u \in C^{1,\alpha^{\prime}} \), hence \( \Omega_i(u, Du) \) are well-defined.  
	Since \( \Omega_i, \, i = 0, \ldots, N \), form a disjoint partition of \( B_1 \), there is a unique \( i_0 \in \{0, \ldots, N\} \)  
	such that \( x_0 \in \Omega_{i_0}(u, Du) \). Thus, we get  
	\begin{equation}
		\bigg(|D\varphi(x_0)|^{\gamma_{i_{0}}(x_{0})}+a(x_{0})|D\varphi(x_0)|^{\sigma_{i_{0}}(x_{0})}\bigg)F \left( D^2\varphi(x_0),x_{0} \right)+ \mathcal{H}({D\varphi(x_{0})}, x_{0})\leq f(x_{0}.)
	\end{equation}
	In particular, we have  
	\begin{equation*}
		\min_{i=0, \ldots, N} \left\{ \bigg(|D\varphi(x_0)|^{\gamma_{i_{0}}(x_{0})}+a(x_{0})|D\varphi(x_0)|^{\sigma_{i_{0}}(x_{0})}\bigg)F \left( D^2\varphi(x_0),x_{0} \right)\right\}+ \mathcal{H}({D\varphi(x_{0})}, x_{0})  \leq \|f\|_{L^\infty(B_1)}.
	\end{equation*}	
	Hence, we have proved that \( u \) is a subsolution of \eqref{bb11}. Similarly, we can see that $u$ is a viscosity supersolution to \eqref{bb12}.
\end{proof}

\section{Local H\"{o}lder regularity of the gradient}\label{section333}
This section is devoted to the proof of Theorem \ref{thm1} concerning the local gradient H\"{o}lder continuity for viscosity solutions to \eqref{11model}. From Lemma \ref{lem2.1}, we know that the $C^{1,\alpha^{\prime}}$ regularity relies on the two differential inequalities \eqref{aa11} and \eqref{aa12}. To this end, we first prove the following interior $C^{1,\alpha^{\prime}}$ regularity.  
\begin{theorem}\label{thm3}
	Let $u\in C(B_{1})$ be a viscosity subsolution to \eqref{aa11} and a supersolution to \eqref{aa12} under assumptions \eqref{11}-\eqref{15}. Then $u\in C_{\rm loc}^{1,\alpha^{\prime}}(B_{1})$ for some $\alpha^{\prime}\in (0,\alpha_{0})\cap \left(0,\frac{1}{1+\gamma_{2}}\right]$ with the following estimate
	\begin{equation*}
		\|u\|_{C^{1, \alpha^{\prime}}(B_{1/2})}\leq
		\begin{cases}
			C_{1}\left(1+\|u\|_{L^{\infty}\left(B_{1}\right)}+\left(\|f\|_{L^{\infty}\left(B_{1}\right)}+\mathcal{K}\right)^{\frac{1}{1+\gamma_{1}}}+\mathcal{M}^{\frac{1}{1+\gamma_{1}-m}}\right),&{\rm if}\;\; 0<m<1+\gamma_{1},\\
			C_{2}\left(1+\|u\|_{L^{\infty}\left(B_{1}\right)}\right), &{\rm if}\;\; m=1+\gamma_{1},
		\end{cases} 
	\end{equation*}
where constant $C_{1}$ depends on $d,\lambda,\Lambda,\alpha^{\prime},m,\gamma_{1}$, and constant $C_{2}$ depends on $d,\lambda,\Lambda,\alpha^{\prime},m,\gamma_{1},\|f\|_{L^{\infty}\left(B_{1}\right)}$, $\mathcal{K}$ and $\mathcal{M}$.
\end{theorem}
We proceed to give a direct consequence of Theorem \ref{thm3}, which will be used in the proof of Theorem \ref{thm2}.
\begin{corollary}\label{coro1}
	Let $u\in C(B_{1})$ be a viscosity subsolution of \eqref{bb11} and a supersolution to \eqref{bb12} under hypotheses \eqref{11}-\eqref{12} and \eqref{14}-\eqref{19}. Then $u\in C_{\rm loc}^{1,\alpha^{\prime}}(B_{1})$ with $\alpha^{\prime}\in (0,\alpha_{0})\cap \left(0,\frac{1}{1+\gamma_{2}}\right]$.
\end{corollary}

\subsection{Compactness of solutions}
In this subsection, let $u$ be a viscosity subsolution to 
\begin{equation}\label{cc11}
	\min_{i=1,2} \left\{ |Du+\xi|^{\gamma_i} F \left( D^2u,x\right)\right\}+ \mathcal{H}({Du+\xi}, x) = \|f\|_{L^\infty(B_1)} \quad  \text{in} \quad B_{1}
\end{equation}
and a viscosity supersolution to
\begin{equation}\label{cc12}
	\max_{i=1,2} \left\{ |Du+\xi|^{\gamma_i} F \left( D^2u,x\right)\right\} + \mathcal{H}({Du+\xi}, x)= -\|f\|_{L^\infty(B_1)} \quad  \text{in} \quad B_{1}
\end{equation}
where $\xi$ is an arbitrary vector in $\rn$. We first establish the following compactness result for the case $0<m\leq \gamma_{1}$. The proof of this Proposition is rather long and delicate, therefore we postpone it to Appendix \ref{sec:app}.
\begin{proposition}\label{prop3.1} 
	Assume \eqref{11}-\eqref{15} hold with $0<m\leq \gamma_{1}$.
	Let $u\in C(B_{1})$ be a viscosity subsolution to \eqref{cc11} and a viscosity supersolution to \eqref{cc12} with $\|u\|_{L^{\infty}(B_{1})}\leq 1$. Then $u\in C_{\rm loc}^{0,\beta_{0}}(B_{1})$ for some $\beta_{0}\in(0,1)$. In addition, there holds that
		$$|u(x)-u(y)|\leq C|x-y|^{\beta_{0}}$$
		for all $x,y\in B_{3/4}$, where $C>0$ is a universal constant.
\end{proposition}

To proceed, we deal with the scenario $\gamma_{1}<m\leq 1+\gamma_{1}$. In this case, we need to impose the additional assumption to restrain the growth of the term $\mathcal{M}\abs{Du+\xi}^{m-\gamma_{1}}$. Its proof is postponed to Appendix \ref{sec:app}.
\begin{proposition}\label{prop3.2}
	Assume \eqref{11}-\eqref{15} hold with $\gamma_{1}<m\leq 1+\gamma_{1}$.
	Let $u\in C(B_{1})$ be a viscosity subsolution to \eqref{cc11} and a viscosity supersolution to \eqref{cc12} with $\|u\|_{L^{\infty}(B_{1})}\leq 1$. There exists a universal constant $\kappa_{0}>0$ such that if
	\begin{equation}\label{control}
		\mathcal{M}\left(\abs{\xi}^{m-\gamma_{1}}+1\right)\leq \kappa
	\end{equation}
	for some $\kappa\leq \kappa_{0}$, then $u\in C_{\rm loc}^{0,\gamma}(B_{1})$ for some $\gamma\in(0,1)$. Furthermore, there exists a universal constant $C>0$ (independent of $\xi$), such that
		$$|u(x)-u(y)|\leq C|x-y|^{\gamma}$$
	for all $x,y\in B_{3/4}$.
\end{proposition}
\subsection{Approximation lemma}
Building upon the previous compactness results, this section is  devoted to proving a key approximation lemma, which plays a paramount role in our forthcoming geometric argument.
\begin{lemma}\label{bijin} 
	Assume \eqref{11}-\eqref{15} hold.
	Let $u\in C(B_{1})$ be a viscosity subsolution to \eqref{cc11} and a viscosity supersolution to \eqref{cc12} with $\|u\|_{L^{\infty}(B_{1})}\leq 1$. Then, for any $\varepsilon>0$, there exists $\delta\in (0,1)$ 
	such that if
	\begin{equation*}
		\max\left\{\|{\rm osc}_{{F}}\|_{L^{\infty}\left(B_{1}\right)},\|f\|_{L^{\infty}(B_{1})},\mathcal{K},\mathcal{M}\left(\abs{\xi}^{(m-\gamma_{1})_{+}}+1\right) \right\}\leq \delta,
	\end{equation*}
	then one can find a function $v\in C_{\rm loc}^{1,\alpha_{0}}(B_{3/4})$ for some $\alpha_{0}\in(0,1)$, satisfying
	\begin{equation*}
		\|u-v\|_{L^{\infty}(B_{1/2})}\leq \varepsilon.
	\end{equation*}
\end{lemma}
\begin{proof}
	We use a contradiction argument. If the claim fails, then there exist $\varepsilon_{0}>0$ and sequences
	of functions $\{F_{n}\}_{n\in \mathbb{N}}$, $\{\mathcal{H}_{n}\}_{n\in \mathbb{N}}$, $\{u_{n}\}_{n\in \mathbb{N}}$ and a sequence of vectors $\{\xi_{n}\}_{n\in \mathbb{N}}$ such that $u_{n}\in C({B_{1}})$ is a viscosity subsolution of 
		\begin{equation}\label{c11}
			\min_{i=1,2} \biggl\{|Du_{n}+\xi_{n}|^{\gamma_i} F_{n} \left( D^2u_{n},x\right)+ \mathcal{H}_{n}({Du_{n}+\xi_{n}}, x)\biggr\} = \frac{1}{n} \quad  \text{in} \quad B_{1}
		\end{equation}
		and a viscosity supersolution to 
		\begin{equation}\label{c12}
			\max_{i=1,2}\biggl\{|Du_{n}+\xi_{n}|^{\gamma_i} F_{n} \left( D^2u_{n},x\right)+ \mathcal{H}_{n}({Du_{n}+\xi_{n}}, x)\biggr\} =- \frac{1}{n} \quad  \text{in} \quad B_{1}
		\end{equation}
	with $ \|u_{n}\|_{L^{\infty}(B_{1})}\leq 1$, where $F_{n}$ satisfy \eqref{12}, $\mathcal{H}_{n}:\mathbb{R}^{d} \times B_{1}\rightarrow \mathbb{R}$ is continuous and there exist constants $\mathcal{K}_{n},\mathcal{M}_{n}>0$ such
		that
		\begin{equation}\label{7398}
			|\mathcal{H}_{n}(t,x)|\leq \mathcal{K}_{n}+\mathcal{M}_{n}|t|^{m} \quad  {\rm for\; every\;}(t,x) \in\mathbb{R}^{d}\times B_{1}.
		\end{equation}
		Moreover,
		\begin{equation}\label{7399}
			\max\left\{\|{\rm osc}_{{F}_{n}}\|_{L^{\infty}\left(B_{1}\right)}, \mathcal{K}_{n},\mathcal{M}_{n}\left(\abs{\xi_{n}}^{(m-\gamma_{1})_{+}}+1\right) \right\}\leq \frac{1}{n}.
		\end{equation}	
	However, for any $v\in C_{\rm loc}^{1,\alpha_{0}}(B_{3/4})$, it holds
	 \begin{equation}\label{340}
			\|u_{n}-v\|_{L^{\infty}(B_{1/2})}>\varepsilon_{0} \quad {\rm for\;any}\;n\in\mathbb{N}.
	\end{equation}
	
	Since $F_{n}$ is uniform $(\lambda,\Lambda)$-elliptic, by applying \eqref{7399} and Arzel${\rm \grave{a}}$-Ascoli theorem, we know that there exists some uniformly $(\lambda,\Lambda)$-elliptic
 operator $F_{\infty}$ (with frozen coefficients) such that $F_{n}\rightarrow F_{\infty}$ locally uniformly in $S^{d}\times B_{1}$, through a subsequence if necessary. In addition, it follows from Propositions \ref{prop3.1} and \ref{prop3.2} that the sequence $\{u_{n}\}_{n\in\mathbb{N}}\subset C_{\rm loc}^{0,\tau_{0}}(B_{1})$ for some $\tau_{0}\in (0,1)$. By applying Arzel${\rm \grave{a}}$-Ascoli theorem again, we conclude that, up to a subsequence,  $u_{n}$ converges locally uniformly in $B_{1}$ to some continuous function $u_{\infty}$ in the $C^{0}$-topology. 

In what follows, we aim to verify that the limiting function $u_{\infty}$ is a viscosity solution to 
\begin{equation}\label{homojie}
	F_{\infty}(D^{2}u_{\infty})=0 \quad {\rm in}\quad  B_{3/4}.
\end{equation}
We only prove that $u_{\infty}$ is a viscosity supersolution, as its subsolution counterpart is entirely analogous. Let $\varphi$ be any test function touching $u_{\infty}$ from below at a point $\overline{x}\in B_{3/4}$, that is,
$$\varphi(\overline{x})=u_{\infty}(\overline{x})\quad {\rm and}\quad \varphi(x)<u_{\infty}(x)\quad {\rm for\; all}\;\;x\neq \overline{x}.$$
Without loss of generality, we assume that $\abs{\overline{x}}=u_{\infty}(\overline{x}) = 0$ and $\varphi$ is a quadratic polynomial, namely,
$$\varphi(x)=\frac{1}{2}\left\langle Mx,x\right\rangle+\left\langle b,x\right\rangle.$$
Since $u_{n}\rightarrow u_{\infty}$ locally uniformly in $B_{1}$, we see that, for $n$ sufficiently large, the polynomial
$$\varphi_{n}(x):=\frac{1}{2}\left\langle M(x-x_{n}),x-x_{n}\right\rangle+\left\langle b,x-x_{n}\right\rangle+u_{n}(x_{n})$$
touches $u_{n}$ from below at $x_{n}$ belonging to a small neighborhood of the origin. Since $u_{n}$ is a viscosity supersolution of \eqref{c12}, we immediately have
\begin{equation}\label{0926}
	\max_{i=1,2}\biggl\{|b+\xi_{n}|^{\gamma_i} F_{n} \left( M,x_{n}\right)\biggr\}+ \mathcal{H}_{n}({b+\xi_{n}}, x_{n})\geq -\frac{1}{n}.
\end{equation}
We assume by contraction that
\begin{equation}\label{maodun111}
F_{\infty}(M)<0.
\end{equation}
Then it follows that
\begin{equation}\label{818}
	F_{n}(M,x_{n})<0\quad {\rm for\;}n\in \mathbb{N}\;{\rm large\; enough}.
\end{equation}
This yields that, for $n$ large enough,
\begin{equation}\label{0927}
	\max_{i=1,2}\biggl\{|b+\xi_{n}|^{\gamma_i} F_{n} \left( M,x_{n}\right)\biggr\}=\min\biggl\{|b+\xi_{n}|^{\gamma_1},|b+\xi_{n}|^{\gamma_2}\biggr\} F_{n} \left( M,x_{n}\right)
\end{equation}
Combining \eqref{0926} with \eqref{0927}, we get
\begin{equation}\label{377}
	F_{n} \left( M,x_{n}\right)\geq  -\frac{1}{\min\biggl\{|b+\xi_{n}|^{\gamma_1},|b+\xi_{n}|^{\gamma_2}\biggr\}}\left(\frac{1}{n}+\mathcal{H}_{n}({b+\xi_{n}}, x_{n})\right)=:\mathcal{I}_{1}.
\end{equation}

If the sequence $\{\xi_{n}\}_{n\in \mathbb{N}}$ is unbounded, take a subsequence, still denoted by $\left\{\xi_{n}\right\}_{n\in\mathbb{N}}$, for which $\abs{\xi_{n}}\rightarrow {\infty}$ as $n\rightarrow\infty$.
Then there exists $n^{\star}\in\mathbb{N}$ so that $\abs{\xi_{n}}\geq2\max\{1,\abs{b}\}$ for all $n\geq n^{\star}$. Then it follows that
\begin{equation}\label{Step1xiajie}
	\abs{b+\xi_{n}}\geq \abs{\xi_{n}}-|b|\geq \frac{1}{2}\abs{\xi_{n}}\geq 1.
\end{equation}
Using \eqref{7398} and \eqref{7399}, in combination with \eqref{Step1xiajie}, $\gamma_{2}\geq \gamma_{1}>0$ and $0<m\leq 1+\gamma_{1}$, we deduce that
\begin{equation*}
	\begin{split}
			\Abs{\mathcal{I}_{1}}&\leq \frac{1}{n}+\frac{\mathcal{K}_{n}+\mathcal{M}_{n}	\abs{b+\xi_{n}}^{m}}{\min\biggl\{|b+\xi_{n}|^{\gamma_1},|b+\xi_{n}|^{\gamma_2}\biggr\}}\\
			&\leq \frac{2}{n}+
			\frac{2^{\gamma_{1}}\mathcal{M}_{n}\abs{b+\xi_{n}}^{m}}{\abs{\xi_{n}}^{\gamma_{1}}}\\
			&\leq \frac{2}{n}+\frac{2^{m+\gamma_{1}}\abs{\xi_{n}}^{m-\gamma_{1}}}{n\left(\abs{\xi_{n}}^{(m-\gamma_{1})_{+}}+1\right)}\left(\Abs{\frac{b}{\xi_{n}}}^{m}+1\right)\\
			&\leq \frac{2}{n}+\frac{2^{m+1+\gamma_{1}}\abs{\xi_{n}}^{m-\gamma_{1}}}{n\left(\abs{\xi_{n}}^{(m-\gamma_{1})_{+}}+1\right)}\\
				&\leq
			\begin{cases}
				\frac{2}{n}+\frac{2^{m+\gamma_{1}}}{n \abs{\xi_{n}}^{\gamma_{1}-m}}& \text{if}\;\; 0< m\leq \gamma_{1} \\
				\frac{2}{n}+\frac{2^{m+\gamma_{1}+1}}{n} &\text{if}\;\; \gamma_{1}<m\leq 1+ \gamma_{1}
			\end{cases}.
	\end{split}
\end{equation*}
Then it follow that $\mathcal{I}_{1}\rightarrow 0$ as $n\rightarrow \infty$. This together with \eqref{377}, we immediately obtain a contradiction to \eqref{maodun111}. 

On the other hand, if the sequence $\{\xi_{n}\}_{n\in\mathbb{N}}$ is bounded, then we may assume  $\xi_{n}\rightarrow \xi_{\infty}$ as $n\rightarrow \infty$ (up to a subsequence). Then it follows that $\xi_{n}+b\rightarrow \xi_{\infty}+b \; {\rm as}\; n\rightarrow \infty.$ At this point, we consider two cases: $\abs{b+\xi_{\infty}}\neq 0$ or $\abs{b+\xi_{\infty}}= 0$. 

In the case of $\abs{b+\xi_{\infty}}\neq 0$. Note that $\Abs{\xi_{n}+b}\geq \frac{1}{2}\Abs{\xi_{\infty}+b}$ for $n$ large enough. Applying \eqref{7398} and \eqref{7399}, in combination with \eqref{377}, $\gamma_{2}\geq \gamma_{1}>0$ and $0<m\leq 1+\gamma_{1}$, we obtain
\begin{equation*}
	\begin{split}
			F_{n} \left( M,x_{n}\right)
			\geq -\frac{2^{1+\gamma_{2}}}{nC_{3}}- \frac{2^{\gamma_{2}}}{nC_{3}} \frac{\abs{b+\xi_{n}}^{m}}{\abs{\xi_{n}}^{(m-\gamma_{1})_{+}}+1}\geq -\frac{2^{\gamma_{2}}}{nC_{3}}\left(2+\abs{b+\xi_{n}}^{m}\right),
	\end{split}
\end{equation*}
where $C_{3}:=\min\left\{|b+\xi_{\infty}|^{\gamma_{1}},|b+\xi_{\infty}|^{\gamma_{2}}\right\}$.
Passing to the limit $n\rightarrow \infty$ in the previous display, we reach a contradiction to \eqref{maodun111}. 

For the latter case where $\Abs{b+\xi_{\infty}}= 0$, there are two situations can occur: $b=-\xi_{\infty}$ with $|b|,\abs{\xi_{\infty}}>0$ or $|b|=\abs{\xi_{\infty}}=0$. \\
{\bf Case 1.} $b=-\xi_{\infty}$ with $|b|,\abs{\xi_{\infty}}>0$. By \eqref{maodun111} and ellipticity, we know that matrix $M$ has at least one positive eigenvalue. Let $\rn=T\oplus Q$ be the orthogonal sum, where $T:={\rm span}\{e_{1},e_{2},...,e_{k}\}$ is the invariant space composed of those eigenvectors corresponding to
positive eigenvalues of $M$, and $Q:=\{y\in \rn:\left\langle y,\eta  \right\rangle=0\;{\rm for\;all}\;\eta\in T\}$.
 Let $\gamma>0$ and
\begin{equation}\label{model3}
	\varphi_{\gamma}(x):=\varphi(x)+\gamma \Abs{P_{T}(x)}=\frac{1}{2}\left\langle Mx,x\right\rangle+\left\langle b,x\right\rangle+\gamma \Abs{P_{T}(x)},
\end{equation}
where $P_{T}$ denotes the orthogonal projection over $T$. Since $u_{n}\rightarrow u_{\infty}$ locally uniformly in $B_{1}$ and $\varphi$ touches $u_{\infty}$ from below at the origin, then, for $\gamma$ small enough, $\varphi_{\gamma}$ touches
$u_{n}$ from below at a point $x_{n}^{\gamma}$ belonging to a small neighborhood of the origin. Moreover, there holds that, up to a subsequence, $x_{n}^{\gamma}\rightarrow x_{*}$ for some $x_{*}\in B_{3/4}$ as $n\rightarrow\infty$. Now we examine two scenarios: $P_{T}\left(x_{n}^{\gamma}\right)=0$ and $P_{T}\left(x_{n}^{\gamma}\right)\neq 0$.

First, we consider $P_{T}\left(x_{n}^{\gamma}\right)=0$.
Notice that
\begin{equation*}
	\tilde{\varphi}_{\gamma}(x):=\varphi(x)+\gamma \left\langle e,P_{T}(x)\right\rangle
\end{equation*}
touches $u_{n}$ from below at $x_{n}^{\gamma}$ for every $e\in \mathbb{S}^{d-1}$ (i.e., $|e|=1$). A direct calculation yields
\begin{equation*}
	D\tilde{\varphi}_{\gamma}(x_{n}^{\gamma})=Mx_{n}^{\gamma}+b+\gamma P_{T}(e),\quad D^{2}\tilde{\varphi}_{\gamma}(x_{n}^{\gamma})=M.
\end{equation*}
We choose $e\in T\cap \mathbb{S}^{d-1}$ such that $P_{T}(e)=e$.
Since $u_{n}$  is a viscosity supersolution of \eqref{c12}, we have
\begin{equation*}
	\begin{split}
		\max_{i=1,2}\biggl\{|Mx_{n}^{\gamma}+b+\gamma e+\xi_{n}|^{\gamma_i} F_{n} \left( M,x_{n}^{\gamma}\right)\biggr\}+ \mathcal{H}_{n}({Mx_{n}^{\gamma}+b+\gamma e+\xi_{n}}, x_{n}^{\gamma})\geq -\frac{1}{n}.
	\end{split}
\end{equation*}
By virtue of \eqref{818}, we obtain that, for $n$ large enough,
\begin{equation*}
	\begin{split}
&\max_{i=1,2}\biggl\{|Mx_{n}^{\gamma}+b+\gamma e+\xi_{n}|^{\gamma_i} F_{n} \left( M,x_{n}^{\gamma}\right)\biggr\}\\
=&\min\biggl\{|Mx_{n}^{\gamma}+b+\gamma e+\xi_{n}|^{\gamma_1},|Mx_{n}^{\gamma}+b+\gamma e+\xi_{n}|^{\gamma_2}\biggr\} F_{n} \left( M,x_{n}\right).
	\end{split}
\end{equation*}
Then it follows that
\begin{equation}\label{0931}
	\begin{split}
		F_{n} \left( M,x_{n}^{\gamma}\right)\geq  -\frac{n^{-1}+\mathcal{H}_{n}({Mx_{n}^{\gamma}+b+\gamma e+\xi_{n}}, x_{n})}{\min\biggl\{|Mx_{n}^{\gamma}+b+\gamma e+\xi_{n}|^{\gamma_1},|Mx_{n}^{\gamma}+b+\gamma e+\xi_{n}|^{\gamma_2}\biggr\}}=:\mathcal{I}_{2}.
	\end{split}
\end{equation}
If $Mx_{*}=0$, then for $n$ sufficiently large, we have
\begin{equation*}
	\Abs{Mx_{n}^{\gamma}+b+\xi_{n}}\leq \frac{\gamma}{2}.
\end{equation*}
This along with the triangle inequality yields
\begin{equation}\label{0932}
	\frac{\gamma}{2}\leq \Abs{Mx_{n}^{\gamma}+b+\gamma e+\xi_{n}}\leq  \frac{3\gamma}{2}.
\end{equation}
Applying \eqref{7398} and \eqref{7399}, in combination with \eqref{0932} and $\gamma_{2}\geq \gamma_{1}>0$, we deduce that
$$|\mathcal{I}_{2}|\leq \frac{2^{1+\gamma_{2}}}{n\gamma^{\gamma_{2}}}+\frac{2^{\gamma_{2}-m}(3\gamma)^{m}}{n\gamma^{\gamma_{2}}}.$$
This along with \eqref{0931} yields 
\begin{equation}\label{0955}
	\begin{split}
		F_{n} \left( M,x_{n}^{\gamma}\right)\geq  -\frac{2^{1+\gamma_{2}}}{n\gamma^{\gamma_{2}}}-\frac{2^{\gamma_{2}-m}(3\gamma)^{m}}{n\gamma^{\gamma_{2}}}.
	\end{split}
\end{equation}
By passing to the limit $n\rightarrow \infty$ in \eqref{0955}, we reach a contradiction with \eqref{maodun111}. \\
On the other hand, if $\Abs{Mx_{*}}>0$, we start off by considering the case in which $T\equiv \rn$ and select $e\in \mathbb{S}^{d-1}$ such that
\begin{equation*}
	\Abs{Mx_{*}+\gamma P_{T}(e)}=\Abs{Mx_{*}+\gamma e}>0.
\end{equation*}
For $n$ large enough, we have
\begin{equation}\label{case1}
	\frac{1}{2}\Abs{Mx_{*}+\gamma e}\leq \Abs{Mx_{n}^{\gamma}+\gamma e}\leq \frac{3}{2}\Abs{Mx_{*}+\gamma e}\quad {\rm and}\quad \abs{\xi_{n}+b}\leq \frac{1}{8}\Abs{Mx_{*}+\gamma e}.
\end{equation}
Furthermore, if $T\neq \rn$, then we choose $e\in \mathbb{S}^{d-1}\cap T^{\perp}$ such that
\begin{equation*}
	\Abs{Mx_{*}+\gamma P_{T}(e)}=\Abs{Mx_{*}}>0.
\end{equation*}
Again for $n$ large enough, we have
\begin{equation}\label{case2}
	\frac{1}{2}\Abs{Mx_{*}}\leq \Abs{Mx_{n}^{\gamma}}\leq \frac{3}{2}\Abs{Mx_{*}}\quad {\rm and}\quad \abs{\xi_{n}+b}\leq \frac{1}{8}\Abs{Mx_{*}}.
\end{equation}
Thus, using either \eqref{case1} or \eqref{case2}, we arrive at
\begin{equation*}
	0<\frac{3}{8}\Abs{Mx_{*}+\gamma P_{T}(e)}\leq\Abs{Mx_{n}^{\gamma}+b+\gamma P_{T}(e)+\xi_{j}}\leq \frac{13}{8}\Abs{Mx_{*}+\gamma P_{T}(e)}.
\end{equation*}
By the same arguments as before, we deduce that
\begin{equation*}
	\begin{split}
		F_{n} \left( M,x_{n}^{\gamma}\right)&\geq  -\frac{n^{-1}+\mathcal{H}_{n}({Mx_{n}^{\gamma}+b+\gamma P_{T}(e)+\xi_{n}}, x_{n})}{\min\biggl\{|Mx_{n}^{\gamma}+b+\gamma P_{T}(e)+\xi_{n}|^{\gamma_1},|Mx_{n}^{\gamma}+b+\gamma P_{T}(e)+\xi_{n}|^{\gamma_2}\biggr\}}\\
		&\geq -\frac{1}{nC_{4}}\left(2+\left(\frac{13}{8}|Mx_{*}+\gamma P_{T}(e)|\right)^{m}\right),
	\end{split}
\end{equation*}
where $C_{4}:=\min\left\{\left(\frac{3}{8}|Mx_{*}+\gamma P_{T}(e)|\right)^{\gamma_{1}},\left(\frac{3}{8}|Mx_{*}+\gamma P_{T}(e)|\right)^{\gamma_{2}}\right\}$.
By passing to the limit $n\rightarrow \infty$, we also obtain a contradiction with \eqref{maodun111}.

Next, let us consider $P_{T}\left(x_{n}^{\gamma}\right)\neq 0$.
Note that $\Abs{P_{T}(x)}$ is smooth and convex in a small neighborhood of $x_{n}^{\gamma}$. Because of $P_{T}$ being a projection, then
\begin{equation}\label{feifuding}
	\Abs{P_{T}(x_{n}^{\gamma})}D\left(\Abs{P_{T}(x_{n}^{\gamma})}\right)=P_{T}(x_{n}^{\gamma})\quad {\rm and} \quad D^{2}\left(\Abs{P_{T}(x_{n}^{\gamma})}\right)\;\,{\rm is \;nonnegative\; definite}.
\end{equation}
Then it follows that
$$D{\varphi}_{\gamma}(x_{n}^{\gamma})= Mx_n^\gamma+ {b} + \gamma D(|P_T(x_n^\gamma)|),\quad D^{2}{\varphi}_{\gamma}(x_{n}^{\gamma})=M+\gamma D^{2}\left(\Abs{P_{T}(x_{n}^{\gamma})}\right).$$
Since $u_{n}$  is a viscosity supersolution of \eqref{c12}, we have
\begin{equation*}
	\max_{i=1,2}\biggl\{|D{\varphi}_{\gamma}(x_{n}^{\gamma})+\xi_{n}|^{\gamma_i} F_{n} \left( D^{2}{\varphi}_{\gamma}(x_{n}^{\gamma}),x_{n}^{\gamma}\right)\biggr\}+ \mathcal{H}_{n}(D{\varphi}_{\gamma}(x_{n}^{\gamma})+\xi_{n}, x_{n}^{\gamma})\geq -\frac{1}{n}.
\end{equation*}
As in the case when $P_{T}\left(x_{n}^{\gamma}\right)= 0$, we can consider the scenarios that $Mx_{*}=0$ and $Mx_{*}\neq 0$. Then we can conclude that
$$F_{\infty}(M+\gamma D^{2}\left(\Abs{P_{T}(x_{*})}\right)\geq 0.$$
By virtue of \eqref{feifuding} and the ellipticity condition on $F_{\infty}$, we derive
$$F_{\infty}(M)\geq F_{\infty}(M+\gamma D^{2}\left(\Abs{P_{T}(x_{*})}\right)\geq 0,$$
which contradicts to \eqref{maodun111}.\\
{\bf Case 2.} $b=\xi_{\infty}=0$. In this case,  its proof is very analogous to Case 1 and so we skip the details here.

At this stage, we have shown
that $u_{\infty}$ is a viscosity supersolution of \eqref{homojie}. It follows from the well-known regularity results \cite[Chapter 5]{Caff1} that $u_{\infty}\in C_{\rm loc}^{1,\alpha_{0}}(B_{3/4})$ for some $\alpha_{0}\in(0,1)$. Finally, taking $v=u_{\infty}$ leads to a contradiction with \eqref{340} for $n$ sufficiently large. This completes the proof.
\end{proof}
\subsection{Proof of Theorems \ref{thm3} and \ref{thm1}}
We start off by establishing the following pointwise regularity estimate. 
\begin{proposition}\label{prop3.4}
	Let $u\in C(B_{1})$ be a viscosity subsolution to \eqref{aa11} and a viscosity supersolution to \eqref{aa12} under assumptions \eqref{11}-\eqref{15}. Let us fix 
	\begin{equation}\label{zhibiao}
		\alpha^{\prime}\in (0,\alpha_{0})\cap \left(0,\frac{1}{1+\gamma_{2}}\right].
	\end{equation}
 There exist universal constants $\mu>0$ and $\rho\in(0,\frac{1}{2})$ such that if
	\begin{equation}\label{811}
		\max\left\{\|{\rm osc}_{{F}}\|_{L^{\infty}\left(B_{1}\right)},\|f\|_{L^{\infty}(B_{1})},\mathcal{K},\mathcal{M}\right\} \leq \mu,
	\end{equation} 
then $u$ is $C^{1,\alpha^{\prime}}$ at the origin with the estimate
\begin{equation}\label{zhengzexing}
	\sup\limits_{x\in B_{r}}\Abs{u(x)-\left(u(0)+Du(0)\cdot x\right)}\leq Cr^{1+\alpha^{\prime}}\quad {\rm for \;all\;} r\in (0,\rho],
\end{equation}
where $C$ is a universal positive constant.
\end{proposition}
\begin{proof}
	We aim to establish the existence of universal constants $0<\rho<\frac{1}{2}$, $C>0$, and a sequence of affine functions 
	 $$\ell_{j}(x)=\mathfrak{a}_{j}+\mathfrak{b}_{j}\cdot x$$
	 with $\left\{\mathfrak{a}_{j}\right\}_{j\in \mathbb{N}}\subset \mathbb{R}$ and $\left\{\mathfrak{b}_{j}\right\}_{j\in \mathbb{N}}\subset \rn$, satisfying, for all $j\in\mathbb{N}$, the following two estimates:
	\begin{equation}\label{711}
		\|u-\ell_{j}\|_{L^{\infty}(B_{\rho^{j}})}\leq \rho^{j(1+\alpha^{\prime})},
	\end{equation}
	\begin{equation}\label{712}
		|\mathfrak{a}_{j}-\mathfrak{a}_{j-1}|+\rho^{j-1}|\mathfrak{b}_{j}-\mathfrak{b}_{j-1}|\leq C\rho^{(j-1)(1+\alpha^{\prime})}.
	\end{equation}
	As usual, we prove this by means of an induction argument.
	For clarity of presentation, we divide the proof into three steps.\\
	{\bf Step 1. Basis of induction.} We first find an affine function $\ell_{1}$ and a universal constant $0<\rho<\frac{1}{2}$ satisfying 
	\begin{equation}\label{651}
		\|u-\ell_{1}\|_{L^{\infty}(B_{\rho})}\leq \rho^{1+\alpha^{\prime}}.
	\end{equation}
	 By Lemma \ref{bijin}, there exists a function $v\in C_{\rm loc}^{1,\alpha_{0}}(B_{3/4})$ such that
	\begin{equation}\label{666}
		\|u-v\|_{L^{\infty}(B_{1/2})}\leq \varepsilon,
	\end{equation}
where $\varepsilon>0$ to be selected later. Remember that the existence of such a function $v$ is guaranteed by Lemma \ref{bijin}, provided that $\mu>0$ is small enough.
	
	According to the $C^{1,\alpha_{0}}$-regularity of $v$, we have
	\begin{equation}\label{zhengzexing34}
		\sup\limits_{x\in B_{\rho}}\Abs{v(x)-\left(v(0)+Dv(0)\cdot x\right)}\leq C\rho^{1+\alpha_{0}}\quad {\rm for \;all\;} \rho\in\left(0,\frac{1}{2}\right),
	\end{equation}
	\begin{equation}\label{0966}
		|v(0)|+|Dv(0)|\leq C,
	\end{equation}
	where $C>0$ is a universal constant. Let us denote
	\begin{equation*}
		\ell_{1}(x):=\mathfrak{a}_{1}+\mathfrak{b}_{1}\cdot x:=v(0)+Dv(0)\cdot x.
	\end{equation*}
	A combination of the triangle inequality  with \eqref{666} and \eqref{zhengzexing34} yields that
	\begin{equation}\label{652}
		\sup\limits_{x\in B_{\rho}}\Abs{u(x)-\ell_{1}(x)}\leq \sup\limits_{x\in B_{\rho}}\Abs{u(x)-v(x)}+\sup\limits_{x\in B_{\rho}}\Abs{v(x)-\ell_{1}(x)}\leq \varepsilon+C\rho^{1+\alpha_{0}}.
	\end{equation}	
	In light of $\alpha^{\prime}<\alpha_{0}$, we take $0<\rho<\frac{1}{2}$ such that
		\begin{equation}\label{653}
	\rho\leq \left(\frac{1}{2C}\right)^{1/(\alpha^{\prime}-\alpha_{0})} \quad {\rm and}\quad \varepsilon:=\frac{1}{2}\rho^{1+\alpha^{\prime}}.
	\end{equation}	
	Consequently, combining \eqref{652} with \eqref{653} yields the desired result \eqref{651}. Once we fix the value of $\varepsilon$ here, the quantity $\delta$ in Lemma \ref{bijin} is determined accordingly. Let $\mu>0$ be small enough such that
\begin{equation}\label{deltaxuanze}
	\mu \left((2C)^{(m-\gamma_{1})_{+}}+1\right)\leq \delta.
\end{equation}	
To conclude this step, let $\mathfrak{a}_{0}=\mathfrak{b}_{0}=0$. Then it follows from \eqref{651} and \eqref{0966} that \eqref{711} and \eqref{712} hold for $j=1$.\\
{\bf Step 2. Induction process}.  Suppose that the conclusion holds true for $j=1,2,...,k$. Now we are going to show that the claim also holds for $j=k+1$. To this end, we introduce an auxiliary function $u_{k}:B_{1}\rightarrow \mathbb{R}$ as
\begin{equation*}
	u_{k}(x):=\frac{u\left(\rho^{k}x\right)-l_{k}\left(\rho^{k}x\right)}{\rho^{k(1+\alpha^{\prime})}}.
\end{equation*}
Then a direct calculation yields that $u_{k}$ is a viscosity subsolution to  
\begin{equation*}
	\min_{i=1,2} \left\{ |Du_{k}+\xi_{k}|^{\gamma_{i}}F_{k}\left( D^2u_{k},x\right)\right\}+ \mathcal{H}_{k}({Du_{k}+\xi_{k}}, x) = \|f_{k}\|_{L^\infty(B_1)} \quad  \text{in} \quad B_{1}
\end{equation*}
and a viscosity supersolution to the equation  
\begin{equation*}
	\max_{i=1,2} \left\{ |Du_{k}+\xi_{k}|^{\gamma_{i}}F_{k}\left( D^2u_{k},x\right)\right\}+ \mathcal{H}_{k}({Du_{k}+\xi_{k}}, x) = -\|f_{k}\|_{L^\infty(B_1)} \quad  \text{in} \quad B_{1}
\end{equation*}
where
\begin{align*}
	{F_{k}}(X,x):=&\rho^{k(1-\alpha^{\prime})}F\left(\rho^{-k(1-\alpha^{\prime})}X,\rho^{k}x\right)b\quad \xi_{k}:=\rho^{-k\alpha^{\prime}}\mathfrak{a}_{k},\\
	\mathcal{H}_{k}(t,x):=&\rho^{k(1-\alpha^{\prime})}\max\left\{\rho^{-k\gamma_{1}\alpha^{\prime}},\rho^{-k\gamma_{2}\alpha^{\prime}}\right\}\mathcal{H}\left(\rho^{-k\alpha^{\prime}}t,\rho^{k}x\right),\\
	\|f_{k}\|_{L^\infty(B_1)}:=&\rho^{k(1-\alpha^{\prime})}\max\left\{\rho^{-k\gamma_{1}\alpha^{\prime}},\rho^{-k\gamma_{2}\alpha^{\prime}}\right\}\|f\|_{L^\infty(B_1)}.
\end{align*}
Note that ${F_{k}}$ is a uniformly $(\lambda,\Lambda)$-elliptic operator. Through a straightforward calculation, we get
\begin{equation*}
	{\rm osc}_{{F_{k}}}(x,0)=\sup\limits_{M\in S^{d}\setminus \{0\}}\frac{\Abs{F\left(\rho^{-k(1-\alpha^{\prime})}M,\rho^{k}x\right)-F\left(\rho^{-k(1-\alpha^{\prime})}M,0\right)}}{\rho^{-k(1-\alpha^{\prime})}\|M\|}={\rm osc}_{{F}}(\rho^{k}x,0),
\end{equation*}
which together with \eqref{12} yields that
\begin{equation*}
	\|{\rm osc}_{{F_{k}}}\|_{L^{\infty}\left(B_{1}\right)}= 	\|{\rm osc}_{{F}}\|_{L^{\infty}\left(B_{\rho^{k}}\right)}\leq \|{\rm osc}_{{F}}\|_{L^{\infty}\left(B_{1}\right)}\leq \mu.
\end{equation*}
By induction assumption, we have
$$\|{u_{k}}\|_{L^{\infty}\left(B_{1}\right)}\leq 1.$$
Applying \eqref{15} and $\rho\in(0,\frac{1}{2})$, in combination with $\gamma_{1}\leq \gamma_{2}$, we arrive at
\begin{equation*}
	|\mathcal{H}_{k}(t,x)|\leq \rho^{k(1-\alpha^{\prime}(1+\gamma_{2}))}\left(\mathcal{K}+\mathcal{M}\rho^{k\alpha^{\prime} m}|t|^{m}\right)=:\mathcal{K}_{k}+\mathcal{M}_{k}|t|^{m}.
\end{equation*} 
According to $\alpha^{\prime}\in \left(0,\frac{1}{1+\gamma_{2}}\right]$ and \eqref{811}, we immediately deduce that
\begin{equation*}
	\|{f_{k}}\|_{L^{\infty}\left(B_{1}\right)}\leq \rho^{k(1-\alpha^{\prime}(1+\gamma_{2}))}\|{f}\|_{L^{\infty}\left(B_{1}\right)}\leq \mu,
\end{equation*}
\begin{equation*}
	\mathcal{K}_{k}=\rho^{k\left(1-\alpha^{\prime}(1+\gamma_{2})\right)}\mathcal{K}\leq \mu.
\end{equation*}
Next we estimate $\mathcal{M}_{k}\left(\abs{\xi_{k}}^{(m-\gamma_{1})_{+}}+1\right)$. Applying \eqref{811} and $\xi_{k}=\rho^{-k\alpha^{\prime}}\mathfrak{b}_{k}$ to obtain
\begin{equation}\label{Hamilliang}
	\mathcal{M}_{k}\left(\abs{\xi_{k}}^{(m-\gamma_{1})_{+}}+1\right)\leq \mu \rho^{k\left(1-\alpha^{\prime}(1+\gamma_{2}-m)\right)-k\alpha^{\prime}(m-\gamma_{1})_{+}}\left(\abs{\mathfrak{b}_{k}}^{(m-\gamma_{1})_{+}}+1\right).
\end{equation}		
If $0<m\leq \gamma_{1}$,
by virtue of \eqref{Hamilliang} and $1-\alpha^{\prime}(1+\gamma_{2}-m)>0$, we get
$$\mathcal{M}_{k}\left(\abs{\xi_{k}}^{(m-\gamma_{1})_{+}}+1\right)\leq 2\mu \rho^{k\left(1-\alpha^{\prime}(1+\gamma_{2}-m)\right)}\leq 2\mu.$$
If $\gamma_{1}<m\leq 1+\gamma_{1}$, then
\begin{align*}
	\mathcal{M}_{k}\left(\abs{\xi_{k}}^{(m-\gamma_{1})_{+}}+1\right)\leq \mu \rho^{k\left(1-\alpha^{\prime}(1+\gamma_{2}-\gamma_{1})\right)}\left(\abs{\mathfrak{b}_{k}}^{m-\gamma_{1}}+1\right)
	\leq \mu\left(\abs{\mathfrak{b}_{k}}^{m-\gamma_{1}}+1\right),
\end{align*}
where we use the fact that $1-\alpha^{\prime}(1+\gamma_{2}-\gamma_{1})>0$ in the last inequality.
It follows from the induction hypothesis \eqref{712} and $\rho<\frac{1}{2}$ that		
\begin{align*}
	|\mathfrak{b}_{k}|\leq |\mathfrak{b}_{0}|+\sum_{j=1}^{k}\abs{\mathfrak{b}_{j}-\mathfrak{b}_{j-1}}\leq C\sum_{j=1}^{k}\rho^{\alpha^{\prime}(j-1)}\leq \frac{C}{1-\rho^{\alpha^{\prime}}}\leq 2C.
\end{align*}
In summary, for $0<m\leq 1+\gamma_{1}$, we have
\begin{align*}
	\mathcal{M}_{k}\left(\abs{\xi_{k}}^{(m-\gamma_{1})_{+}}+1\right)\leq \mu \left((2C)^{(m-\gamma_{1})_{+}}+1\right).
\end{align*}
As a consequence, combining the information above with \eqref{deltaxuanze}, we arrive at
\begin{equation*}
	\max\left\{\|{\rm osc}_{{F_{k}}}\|_{L^{\infty}\left(B_{1}\right)},\|{f_{k}}\|_{L^{\infty}\left(B_{1}\right)},\mathcal{K}_{k},\mathcal{M}_{k}\left(\abs{\xi_{k}}^{(m-\gamma_{1})_{+}}+1\right)\right\}\leq \delta.
\end{equation*}	
At this moment, the assumptions in Lemma \ref{bijin} are satisfied. Thus, we can apply the conclusion from Step 1
 to obtain
\begin{equation*}
	\|u_{k}-\tilde{\ell}\|_{L^{\infty}(B_{\rho})}\leq \rho^{1+\alpha^{\prime}},
\end{equation*}
where $\tilde{ell}(x)$ is an affine function of the form
$\tilde{\ell}(x)=\tilde{\mathfrak{a}}+\tilde{\mathfrak{b}}\cdot x$ with $|\tilde{\mathfrak{a}}|+|\tilde{\mathfrak{b}}|\leq C$. To proceed, we define 
\begin{equation*}
	\ell_{k+1}(x):=\mathfrak{a}_{k+1}+\mathfrak{b}_{k+1}\cdot x,
\end{equation*}
where
\begin{equation*}
	\mathfrak{a}_{k+1}:=\mathfrak{a}_{k}+\rho^{k(1+\alpha^{\prime})}\tilde{\mathfrak{a}}\quad {\rm and}\quad     \mathfrak{b}_{k+1}:=\mathfrak{b}_{k}+\rho^{k\alpha^{\prime}}\tilde{\mathfrak{b}}.
\end{equation*}
Scaling back, we obtain
\begin{equation*}
	\|u-\ell_{k+1}\|_{L^{\infty}(B_{\rho^{k+1}})}=\rho^{k(1+\alpha^{\prime})}\|u_{k}-\tilde{\ell}\|_{L^{\infty}(B_{\rho})}
	\leq \rho^{(k+1)(1+\alpha^{\prime})}
\end{equation*}
and
\begin{equation*}
	|\mathfrak{a}_{k+1}-\mathfrak{a}_{k}|+\rho^{k}|\mathfrak{b}_{k+1}-\mathfrak{b}_{k}|=\rho^{k(1+\alpha^{\prime})}\left(\abs{\tilde{\mathfrak{a}}}+\abs{\tilde{\mathfrak{b}}}\right)\leq C\rho^{k(1+\alpha^{\prime})}.
\end{equation*}
Hence, we have proven that \eqref{711} and \eqref{712} hold for all $j\in \mathbb{N}$.\\
{\bf Step 3. Convergence analysis}. From Step 2, we know that $\{\mathfrak{a}_{j}\}_{j\in\mathbb{N}}\subset \mathbb{R}$, $\{\mathfrak{b}_{j}\}_{j\in\mathbb{N}}\subset \rn$ are Cauchy sequences, hence, they converge, that is
\begin{equation*}
	\mathfrak{a}_{j}\rightarrow\overline{\mathfrak{a}} \quad {\rm and}\quad 
	\quad \mathfrak{b}_{j}\rightarrow\overline{\mathfrak{b}}
\end{equation*}	
with	
\begin{align}\label{511}
	|\mathfrak{a}_{j}-\overline{\mathfrak{a}}|\leq  \frac{C\rho^{j(1+\alpha^{\prime})}}{1-\rho^{(1+\alpha^{\prime})}}\quad {\rm and}\quad 	|\mathfrak{b}_{j}-\overline{\mathfrak{b}}|\leq  \frac{C\rho^{j\alpha^{\prime}}}{1-\rho^{\alpha^{\prime}}}.
\end{align}	

Finally, given any $0<r\leq \rho$, we can find $j\in \mathbb{N}$ such that $\rho^{j+1}<r\leq \rho^{j}$. Then, for $\overline{\ell}(x):=\overline{\mathfrak{a}}+\overline{\mathfrak{b}}\cdot x$, we arrive at
\begin{align*}
	\|u-\overline{\ell}\|_{L^{\infty}(B_{r})}&\leq \|u-\overline{\ell}\|_{L^{\infty}(B_{\rho^{j}})}\\
	&\leq \|u-{\ell_{j}}\|_{L^{\infty}(B_{\rho^{j}})}+\|\ell_{j}-\overline{\ell}\|_{L^{\infty}(B_{\rho^{j}})}\\
	&\leq  \rho^{j(1+\alpha^{\prime})}+|\mathfrak{a}_{j}-\overline{\mathfrak{a}}|+\rho^{j}|\mathfrak{b}_{j}-\overline{\mathfrak{b}}|\\
	&\leq \rho^{j(1+\alpha^{\prime})}\left(1+\frac{2C}{1-\rho^{\alpha^{\prime}}}\right)\\
	&\leq \frac{1}{\rho^{1+\alpha^{\prime}}}\left(1+\frac{2C}{1-\rho^{\alpha^{\prime}}}\right)r^{1+\alpha^{\prime}}\\
	&\leq Cr^{1+\alpha^{\prime}}.
\end{align*}
This implies that $u$ is $C^{1,\alpha}$ at $0$. The proof is complete.
\end{proof}
We now conclude the proof of Theorem \ref{thm3}.
\begin{proof}[Proof of Theorem \ref{thm3}]
	Initially, we exploit the scaling features of equation to reduce the problem to a smallness regime. That is, without loss of generality, we explicitly verify that it is possible to suppose that
	\begin{equation*}
		\|{u}\|_{L^{\infty}\left(B_{1}\right)}\leq 1 \quad {\rm and}\quad \max\left\{\|{\rm osc}_{{F}}\|_{L^{\infty}\left(B_{1}\right)},\|{f}\|_{L^{\infty}\left(B_{1}\right)},\mathcal{K},\mathcal{M}\right\}\leq \mu
	\end{equation*}
	for constant $0<\mu<1$ coming from Proposition \ref{prop3.4}. In fact, let $u$ be a viscosity subsolution of \eqref{aa11} and a supersolution to \eqref{aa12}. Define
	$$\tilde{u}(x):=\frac{u(\tau x)}{K},\quad x\in B_{1},$$
	where $K,\tau$ are positive constants to be determined later. Then a straightforward computation yields that $\tilde{u}$ is a viscosity subsolution to 
	\begin{equation*}
		\min_{i=1,2} \left\{ |D\tilde{u}|^{\gamma_i} \tilde{F}(D^2 \tilde{u}, x)\right\}+ \tilde{\mathcal{H}}(D\tilde{u},x) = \|\tilde{f}\|_{L^\infty(B_1)} \quad  \text{in} \quad B_{1}
	\end{equation*}
	and a viscosity supersolution to 
	\begin{equation*}
		\max_{i=1,2} \left\{ |D\tilde{u}|^{\gamma_i} \tilde{F}(D^2 \tilde{u}, x)\right\}+ \tilde{\mathcal{H}}(D\tilde{u},x)= -\|\tilde{f}\|_{L^\infty(B_1)} \quad  \text{in} \quad B_{1},
	\end{equation*}
	where
	\begin{align*}
		\tilde{F}(X,x):=&\frac{\tau^{2}}{K}F\left(\frac{K}{\tau^{2}}X,\tau x\right),\\
		\tilde{\mathcal{H}}(t,x):=&\max\left\{\frac{\tau^{2+\gamma_{1}}}{K^{1+\gamma_{1}}},\frac{\tau^{2+\gamma_{2}}}{K^{1+\gamma_{2}}}\right\}\mathcal{H}\left(\frac{K}{\tau}t,\tau x\right),\\
		\|\tilde{f}\|_{L^\infty(B_1)}:=&\max\left\{\frac{\tau^{2+\gamma_{1}}}{K^{1+\gamma_{1}}},\frac{\tau^{2+\gamma_{2}}}{K^{1+\gamma_{2}}}\right\}\|{f}\|_{L^\infty(B_1)}.
	\end{align*}
	Note that $\tilde{F}$ is still a uniformly $(\lambda,\Lambda)$-elliptic operator, and a direct calculation yields that
	\begin{equation*}
		\|{\rm osc}_{\tilde{F}}\|_{L^{\infty}\left(B_{1}\right)}= \|{\rm osc}_{{F}}\|_{L^{\infty}\left(B_{\tau}\right)}\leq C_{F}\tau^{\theta}.
	\end{equation*}
Applying \eqref{15} to deduce
	\begin{equation*}
		|\tilde{\mathcal{H}}(t,x)|\leq \max\left\{\frac{\tau^{2+\gamma_{1}}}{K^{1+\gamma_{1}}},\frac{\tau^{2+\gamma_{2}}}{K^{1+\gamma_{2}}}\right\}\bigg(\mathcal{K}+\mathcal{M}\left(\frac{K}{\tau}\right)^{m}|t|^{m}\bigg)=:\tilde{\mathcal{K}}+\tilde{\mathcal{M}}|t|^{m}.
	\end{equation*}
	Now, we select
	\begin{equation*}
		K:=
		\begin{cases}
		1+\|u\|_{L^{\infty}\left(B_{1}\right)}+\left(\|f\|_{L^{\infty}\left(B_{1}\right)}+\mathcal{K}\right)^{\frac{1}{1+\gamma_{1}}}+\mathcal{M}^{\frac{1}{1+\gamma_{1}-m}} & \text{for} \;\; m<1+\gamma_{1}, \\
		1+\|u\|_{L^{\infty}\left(B_{1}\right)} &\text{for} \;\; m=1+\gamma_{1},	
		\end{cases}
	\end{equation*}
	and
	\begin{equation*}
		\tau:=\begin{cases}
			\min\left\{1,\left(\frac{\delta}{C_{F}}\right)^{\frac{1}{\theta}},\delta^{\frac{1}{2+\gamma_{1}}},\delta^{\frac{1}{2+\gamma_{1}-m}}\right\} & \text{for} \;\; m<1+\gamma_{1}, \\
			\min\left\{1,\left(\frac{\delta}{C_{F}}\right)^{\frac{1}{\theta}},\left(\frac{\delta}{\|f\|_{L^{\infty}(B_{1})}+\mathcal{K}}\right)^{\frac{1}{2+\gamma_{1}}},\frac{\delta}{\mathcal{M}}\right\} &\text{for}\;\; m=1+\gamma_{1}.
			\end{cases}
	\end{equation*}
	With such choice, we arrive at
	\begin{equation*}
		\|\tilde{u}\|_{L^{\infty}\left(B_{1}\right)}\leq 1\quad {\rm and }\quad
		\max\left\{\|{\rm osc}_{\tilde{F}}\|_{L^{\infty}\left(B_{1}\right)},\|\tilde{f}\|_{L^{\infty}\left(B_{1}\right)},\tilde{\mathcal{K}},\tilde{\mathcal{M}}\right\}
		\leq \delta.
	\end{equation*}	
	
	 Therefore, we can employ Proposition \ref{prop3.4} to obtain $u$ is $C^{1,\alpha^{\prime}}$ at $0$. Then by standard 
	translation and covering arguments, we can conclude that $u\in C_{\rm loc}^{1,\alpha^{\prime}}(B_{1})$.
	The proof is complete.
\end{proof}	
\begin{proof}[Proof of Corollary \ref{coro1}]
	This proof is exactly identical to the proof of Theorem \ref{thm3} by using the standard argument.                                                                                                                               
\end{proof}
We conclude this subsection with the proof of Theorem \ref{thm1}.
\begin{proof}[Proof of Theorem \ref{thm1}]
	Applying Lemma \ref{lem2.1} and Theorem \ref{thm3}, we immediately obtain that $u\in C_{\rm loc}^{1,\alpha^{\prime}}(B_1)$ and the desired estimate. The proof is complete.
\end{proof}

\section{Improved pointwise H\"{o}lder regularity of the gradient}\label{section444}
\subsection{Reduction of the problem}
In order to prove Theorem \ref{thm2}, we first show that a simple rescaling reduces the proof of the problem to the case that 
\begin{equation}\label{323}
		\|{u}\|_{L^{\infty}\left(B_{1}\right)}\leq 1\quad {\rm and }\quad \max\left\{\|{\rm osc}_{{F}}\|_{L^{\infty}\left(B_{1}\right)},\|f\|_{L^{\infty}(B_{1})},\mathcal{K},\mathcal{M}\right\}\leq \iota
\end{equation}
for some small constant $\iota$ which will be chosen later. 
\begin{proposition}\label{prop4.1}
	In order to prove Theorem \ref{thm2}, it is enough to prove that
	\begin{equation*}
		[u]_{C^{1,\alpha}(0)}\leq C
	\end{equation*}
assuming \eqref{323}.
\end{proposition}
\begin{proof}
	Initially, from Lemma \ref{lem2.2}, we know that $u$ is a viscosity subsolution to \eqref{bb11} and a viscosity supersolution to \eqref{bb12}. For any fixed $x_{0}\in B_{1/2}$, we define $\tilde{u}:B_{1}\rightarrow \mathbb{R}$ by
	$$\tilde{u}(x):=\frac{u(x_{0}+\varsigma x)}{L},$$
	where 
	\begin{equation*}
		L:=
		\begin{cases}
			1+\|u\|_{L^{\infty}\left(B_{1}\right)}+\left(\|f\|_{L^{\infty}\left(B_{1}\right)}+\mathcal{K}\right)^{\frac{1}{1+\gamma_{1}}}+\mathcal{M}^{\frac{1}{1+\gamma_{1}-m}} & \text{for} \;\; m<1+\gamma_{1}, \\
			1+\|u\|_{L^{\infty}\left(B_{1}\right)} &\text{for} \;\; m=1+\gamma_{1},	
		\end{cases}
	\end{equation*}
	\begin{equation*}
		\varsigma:=\begin{cases}
			\min\left\{\frac{1}{2},\left(\frac{\iota}{C_{F}}\right)^{\frac{1}{\theta}},\iota^{\frac{1}{2+\gamma_{1}}},\iota^{\frac{1}{2+\gamma_{1}-m}}\right\} & \text{for} \;\; m<1+\gamma_{1}, \\
			\min\left\{\frac{1}{2},\left(\frac{\iota}{C_{F}}\right)^{\frac{1}{\theta}},\left(\frac{\iota}{\|f\|_{L^{\infty}(B_{1})}+\mathcal{K}}\right)^{\frac{1}{2+\gamma_{1}}},\frac{\iota}{\mathcal{M}}\right\} &\text{for}\;\; m=1+\gamma_{1}.
		\end{cases}
	\end{equation*}
 Then $\tilde{u}$ is a viscosity subsolution to  
	\begin{equation*}
		\min_{i=0,1,\ldots,N} \left\{ \bigg(|D\tilde{u}|^{\tilde{\gamma}_i(x)}+\tilde{a}(x)|D\tilde{u}|^{\tilde{\sigma}_i(x)}\bigg) \tilde{F}(D^2 \tilde{u}, x)\right\}+ \tilde{\mathcal{H}}(D\tilde{u},x) = \|\tilde{f}\|_{L^\infty(B_1)} \quad  \text{in} \quad B_{1}
	\end{equation*}
	and a viscosity supersolution to 
		\begin{equation*}
		\max_{i=0,1,\ldots,N} \left\{ \bigg(|D\tilde{u}|^{\tilde{\gamma}_i(x)}+\tilde{a}(x)|D\tilde{u}|^{\tilde{\sigma}_i(x)}\bigg) \tilde{F}(D^2 \tilde{u}, x)\right\}+ \tilde{\mathcal{H}}(D\tilde{u},x) = -\|\tilde{f}\|_{L^\infty(B_1)} \quad  \text{in} \quad B_{1},
	\end{equation*}
	where
	\begin{align*}
			\tilde{F}(X,x):=&\frac{\varsigma^{2}}{L}F\left(\frac{L}{\varsigma^{2}}X,x_{0}+\varsigma x\right),\\
		\tilde{\gamma}_{i}(x):=&\gamma_{i}(x_{0}+\varsigma x),\quad \tilde{\sigma}_{i}(x):=\sigma_{i}(x_{0}+\varsigma x), \\
		\tilde{a}(x):=&\left(\frac{L}{\varsigma}\right)^{-({\tilde{\gamma}_{i}(x)-\tilde{\sigma}_{i}(x)})}a(x_{0}+\varsigma x),\\
	\tilde{\mathcal{H}}(t,x):=&\max_{i=0,1,\ldots,N}\left\{\frac{\varsigma^{2+\tilde{\gamma}_{i}(x)}}{L^{1+\tilde{\gamma}_{i}(x)}}\right\}\mathcal{H}\left(\frac{L}{\varsigma}t,x_{0}+\varsigma x\right),
		\\
		\tilde{f}(x):=&\max_{i=0,1,\ldots,N}\left\{\frac{\varsigma^{2+\tilde{\gamma}_{i}(x)}}{L^{1+\tilde{\gamma}_{i}(x)}}\right\}\|{f}\|_{L^\infty(B_1)}.
	\end{align*}
	Note that $\tilde{F}$ is still a uniformly $(\lambda,\Lambda)$-elliptic operator and $0<\gamma_{1}\leq \tilde{\gamma}_{i}(x)\leq \tilde{\sigma}_{i}(x)\leq \gamma_{2}<\infty$, $i=0,1,\ldots,N$. In addition,
	\begin{equation*}
		\|{\rm osc}_{\tilde{F}}\|_{L^{\infty}\left(B_{1}\right)}= \|{\rm osc}_{{F}}(\cdot,x_{0})\|_{L^{\infty}\left(B_{\varsigma(x_{0})}\right)}\leq C_{F}\iota^{\theta}.
	\end{equation*}
	Applying \eqref{15}, in combination with $L>1$, $\varsigma\leq \frac{1}{2}$ and $\gamma_{1}\leq \tilde{\gamma}_{i}(x)\leq \tilde{\sigma}_{i}(x)\leq \gamma_{2}$, we deduce that
	\begin{equation*}
		\begin{split}
			|\tilde{\mathcal{H}}(t,x)|&\leq \max_{i=0,1,\cdots,N}\left\{\frac{\varsigma^{2+\tilde{\gamma}_{i}(x)}}{L^{1+\tilde{\gamma}_{i}(x)}}\right\}\bigg(\mathcal{K}+\mathcal{M}\left(\frac{L}{\varsigma}\right)^{m}|t|^{m}\bigg)\\
			&\leq \frac{\varsigma^{2+\gamma_{1}}}{L^{1+\gamma_{1}}}\bigg(\mathcal{K}+\mathcal{M}\left(\frac{L}{\varsigma}\right)^{m}|t|^{m}\bigg)\\
			&=:\tilde{\mathcal{K}}+\tilde{\mathcal{M}}|t|^{m}.
		\end{split}
	\end{equation*}
	According to the choices of $L$ and $\varsigma$, we immediately arrive at
	\begin{equation*}
		\|\tilde{u}\|_{L^{\infty}\left(B_{1}\right)}\leq 1\quad {\rm and }\quad
		\max\left\{\|{\rm osc}_{\tilde{F}}\|_{L^{\infty}\left(B_{1}\right)},\|\tilde{f}\|_{L^{\infty}\left(B_{1}\right)},\tilde{\mathcal{K}},\tilde{\mathcal{M}}\right\}
		\leq \iota.
	\end{equation*}	
	Therefore, if
	\begin{equation*}
		[\tilde{u}]_{C^{1,\alpha}(0)}\leq C,
	\end{equation*}
by scaling back to $u$, we get 
\begin{equation*}
	[{u}]_{C^{1,\alpha}(x_0)}\leq CL,
\end{equation*}
 which concludes the proof.	
\end{proof}

\subsection{Gradient approximation}
 With the aid of the regularity result from Corollary \ref{coro1}, we can show the following approximation lemma.
\begin{lemma}\label{lem5.1}
	Assume \eqref{11}-\eqref{12} and \eqref{14}-\eqref{19} hold. Let $u$ be a viscosity subsolution to \eqref{bb11} and a viscosity supersolution to \eqref{bb12}
	with $\|u\|_{L^{\infty}(B_{1})}\leq 1$. Given $\eta>0$, there exists constant $\iota\in (0,1)$  
	such that if
	\begin{equation}\label{smallcondition}
		\max\left\{\|{\rm osc}_{{F}}\|_{L^{\infty}\left(B_{1}\right)},\|f\|_{L^{\infty}(B_{1})},\mathcal{K},\mathcal{M}\right\}\leq \iota,
	\end{equation}
	then there exists a function $v\in C_{\rm loc}^{1,\alpha_{0}}(B_{3/4})$
such that
	\begin{equation*}
		\max\left\{\|u-v\|_{L^{\infty}(B_{1/2})},\|Du-Dv\|_{L^{\infty}(B_{1/2})}\right\}\leq \eta.
	\end{equation*}
\end{lemma}
\begin{proof}
	We prove by contradiction. Suppose that there exist $\eta_{0}>0$ and sequences
	of functions $\left\{F_{n}\right\}_{n\in \mathbb{N}}$, $\{\mathcal{H}_{n}\}_{n\in \mathbb{N}}$, $\left\{\gamma_{i}^{n}\right\}_{n\in \mathbb{N}}$, $\left\{\sigma_{i}^{n}\right\}_{n\in \mathbb{N}}$, $\left\{a_{n}\right\}_{n\in \mathbb{N}}$,  $\{u_{n}\}_{n\in \mathbb{N}}$ such that $u_{n}\in C({B_{1}})$ is a viscosity subsolution to 
	\begin{equation*}
		\min_{i=0,1,\ldots,N} \left\{ \bigg(|Du_{n}|^{\gamma_i^{n}(x)}+a_{n}(x)|Du_{n}|^{\sigma_i^{n}(x)} \bigg)F_{n} \left( D^2u_{n},x\right)+ \mathcal{H}_{n}({Du_{n}}, x)\right\} = \frac{1}{n}\quad  \text{in} \quad B_{1}
	\end{equation*}
	and a viscosity supersolution to 
	\begin{equation*}
		\max_{i=0,1,\ldots,N} \left\{ \bigg(|Du_{n}|^{\gamma_i^{n}(x)}+a_{n}(x)|Du_{n}|^{\sigma_i^{n}(x)} \bigg)F_{n} \left( D^2u_{n},x\right)+ \mathcal{H}_{n}({Du_{n}}, x)\right\} = -\frac{1}{n}\quad  \text{in} \quad B_{1}
	\end{equation*}
	with $\|u\|_{L^{\infty}(B_{1})}\leq 1$, where $F_{n}$ satisfy \eqref{12}, $0<\gamma_{1}\leq \gamma_{i}^{n}(x)\leq \sigma_{i}^{n}(x)\leq \gamma_{2}<\infty$, $i=0,1,\ldots,N$, $0\leq a_{n}\in C(B_{1})$, $\mathcal{H}_{n}:\mathbb{R}^{d} \times B_{1}\rightarrow \mathbb{R}$ is continuous and there exist constants $\mathcal{K}_{n},\mathcal{M}_{n}>0$ such
	that
	\begin{equation*}
		|\mathcal{H}_{n}(t,x)|\leq \mathcal{K}_{n}+\mathcal{M}_{n}|t|^{m} \quad  {\rm for\; every\;}(t,x) \in\mathbb{R}^{d}\times B_{1};
	\end{equation*}
	and
	\begin{equation}\label{xiao22}
		\max\left\{\|{\rm osc}_{{F}_{n}}\|_{L^{\infty}\left(B_{1}\right)}, \mathcal{K}_{n},\mathcal{M}_{n}\right\}\leq \frac{1}{n}.
	\end{equation}	
Nonetheless, for any function $v\in C_{\rm loc}^{1,\alpha_{0}}(B_{3/4})$, it holds
\begin{equation}\label{contraction111}
	\max\left\{\|u_{n}-v\|_{L^{\infty}(B_{1/2})},\|Du_{n}-Dv\|_{L^{\infty}(B_{1/2})}\right\}> \eta_{0} \quad {\rm for\;any}\;n\in\mathbb{N}.
\end{equation}

	By uniformly $(\lambda,\Lambda)$-ellipticity and \eqref{xiao22}, there exists some uniformly $(\lambda,\Lambda)$-elliptic
operator $F_{\infty}$ (with frozen coefficients) such that $F_{n}\rightarrow F_{\infty}$ locally uniformly in $S^{d}\times B_{1}$, through a subsequence if necessary. We know from Corollary \ref{coro1} that the sequence $\{u_{n}\}_{n\in\mathbb{N}}\subset C_{\rm loc}^{1,\alpha^{\prime}}(B_{1})$ for some $\alpha^{\prime}\in (0,1)$. By applying Arzel${\rm \grave{a}}$-Ascoli theorem, up to a subsequence, we know that $u_{n}$ converges locally uniformly in $B_{1}$ to some continuous function $u_{\infty}$ in the $C^{1}$-topology. Next, by arguing as in Lemma \ref{bijin}, we can 
conclude that $u_{\infty}$ is a viscosity solution of 
	$$F_{\infty}(D^{2}u_{\infty})=0 \quad {\rm in}\;\; B_{3/4}.$$
	Finally, taking $v=u_{\infty}$, we reach a contradiction with \eqref{contraction111} for $n$ sufficiently large. This completes the proof of the desired result.                                     	
\end{proof}
\subsection{Improved Oscillation-Type Estimate}
To begin with, by using the approximation with $F$-harmonic function, we obtain an oscillation estimate for solutions $u$ to \eqref{11model} near the critical set $\left\{x:Du(x)=0\right\}$.
\begin{lemma}\label{lem5.21}
	Under the assumptions of Lemma \ref{lem5.1}, for every $0<\vartheta<{\alpha}_{0}$, there exists a universal constant $0<\varrho<\frac{1}{2}$ such that 
	\begin{equation}\label{911}
		\sup_{B_{\varrho}}\Abs{u(x)-u(0)}\leq \varrho^{1+\vartheta}+|Du(0)|\varrho.
	\end{equation}
\end{lemma}
\begin{proof}
	Let $\eta>0$ to be ﬁxed a posteriori. From Lemma \ref{lem5.1}, we know that there exists $\iota>0$ such that whenever
	\begin{equation*}
		\max\left\{\|{\rm osc}_{{F}}\|_{L^{\infty}\left(B_{1}\right)},\|f\|_{L^{\infty}(B_{1})},\mathcal{K},\mathcal{M}\right\}\leq \iota,
	\end{equation*}
	then there exist a function $v\in C_{\rm loc}^{1,\alpha_{0}}(B_{3/4})$ satisfying
	\begin{equation}\label{0981}
		\max\left\{\|u-v\|_{L^{\infty}(B_{1/2})},\|Du-Dv\|_{L^{\infty}(B_{1/2})}\right\}\leq \eta.
	\end{equation}
	Since $v\in C_{\rm loc}^{1,\alpha_{0}}(B_{3/4})$, there exists a universal constant $C>0$ such that 
	\begin{equation*}
		\sup_{B_{\varrho}}\Abs{v(x)-v(0)-Dv(0)\cdot x}\leq C\varrho^{1+{\alpha}_{0}} \quad {\rm for\; all}\;\varrho\in\left(0,\frac{1}{2}\right).
	\end{equation*}
	This together with \eqref{0981} immediately yields that
	\begin{equation*}
		\begin{split}
			\sup_{B_{\varrho}}\Abs{u(x)-u(0)-Du(0)\cdot x}\leq& \sup_{B_{\varrho}}\Abs{u(x)-v(x)}+\sup_{B_{\varrho}}\Abs{v(x)-v(0)-Dv(0)\cdot x}\\
			&+\sup_{B_{\varrho}}\Abs{u(0)-v(0)+(Dv(0)-Du(0))\cdot x}\\
			\leq& 3\eta+C\varrho^{1+{\alpha}_{0}}\leq \varrho^{1+\vartheta}
		\end{split}
	\end{equation*}
	as long as we make the following universal choices
	\begin{equation*}
		\varrho\in \left(0,\min\left\{\frac{1}{2},\delta_{1},\left(\frac{1}{2C}\right)^{{\alpha}_{0}-\alpha}\right\}\right)\quad {\rm and}\quad \eta\in \left(0,\frac{1}{6}\varrho^{1+\vartheta}\right),
	\end{equation*}
where constant $\delta_{1}>0$ comes from Remark \ref{rmk1}. Therefore, we obtain \eqref{911} and finish the proof.                                                       
\end{proof}
Next we iterate the previous estimate to control the oscillation of the solutions in dyadic balls. 
\begin{lemma}\label{lem5.31}
	Under the assumptions of Lemma \ref{lem5.21}, there exists a non-decreasing sequences $\left\{\alpha_{k}\right\}_{k\in\mathbb{N}}$ such that
	\begin{equation}\label{12345}
		\sup_{B_{\varrho^{k}}}\Abs{u(x)-u(0)}\leq \varrho^{k(1+\alpha_{k})}+|Du(0)|\sum_{i=0}^{k-1}\varrho^{k+i\alpha_{k}}
	\end{equation}
	for all $k\in\mathbb{N}$, where constant $\varrho\in (0,\frac{1}{2})$ is given in Lemma \ref{lem5.21}.
\end{lemma}
\begin{proof}
	The proof follows from an induction argument. Define the non-decreasing sequence
	\begin{equation}\label{11feidijianzhibiaoxulie}
		\alpha_{k}:=\min_{i=0,1,\ldots, N}\left\{\alpha_{0}^{-},\min_{B_{\varrho^{k}}}\frac{1}{1+\gamma_{i}(x)}\right\},
	\end{equation}
	which converges to the number 
	\begin{equation}
		\alpha:=\min_{i=0,1,\ldots, N}\left\{\alpha_{0}^{-},\frac{1}{1+\gamma_{i}(0)}\right\},
	\end{equation}
	Clearly, \eqref{12345} immediately holds for $k=1$ by Lemma \ref{lem5.21}. Suppose that \eqref{12345} holds for $n=1,2,...,k$. Our goal is to show that \eqref{12345} also holds for $n=k+1$. To this end, we introduce an auxiliary function $u_{k}:B_{1}\rightarrow \mathbb{R}$ as
	\begin{equation*}
		u_{k}(x):=\frac{u\left(\varrho^{k}x\right)-u(0)}{\mathcal{A}_{k}}
	\end{equation*}
	with $\mathcal{A}_{k}:=\varrho^{k(1+\alpha_{k})}+|Du(0)|\sum_{i=0}^{k-1}\varrho^{k+i\alpha_{k}}$. We can readily check that $u_{k}$ is a viscosity subsolution to  
	\begin{equation*}
		\min_{i=0,1,\ldots,N} \left\{ \bigg(|Du_{k}|^{\gamma_{ik}(x)}+a_{k}(x)|Du_{k}|^{\sigma_{ik}(x)} \bigg)F_{k}\left( D^2u_{k},x\right)\right\}+ \mathcal{H}_{k}({Du_{k}}, x) = \|f_{k}\|_{L^\infty(B_1)} \quad  \text{in} \quad B_{1}
	\end{equation*}
	and a viscosity supersolution to 
	\begin{equation*}
		\max_{i=0,1,\ldots,N} \left\{ \bigg(|Du_{k}|^{\gamma_{ik}(x)}+a_{k}(x)|Du_{k}|^{\sigma_{ik}(x)} \bigg)F_{k}\left( D^2u_{k},x\right)\right\}+ \mathcal{H}_{k}({Du_{k}}, x) = -\|f_{k}\|_{L^\infty(B_1)} \quad  \text{in} \quad B_{1},
	\end{equation*}
	where
	\begin{align*}
		{F_{k}}(X,x):=&\frac{\varrho^{2k}}{\mathcal{A}_{k}}F\left(\frac{\mathcal{A}_{k}}{\varrho^{2k}}X,\varrho^{k}x\right),\\
		\gamma_{ik}(x):=&\gamma_{i}(\varrho^{k}x),\quad \sigma_{ik}(x):=\sigma_{i}(\varrho^{k}x), \quad
		{a_{k}}(x):=\left(\frac{\varrho^{k}}{\mathcal{A}_{k}}\right)^{\gamma_{ik}(x)-\sigma_{ik}(x)}a(\varrho^{k} x),\\
		\mathcal{H}_{k}(t,x):=&\max_{i=0,1,\ldots,N}\left\{\frac{\varrho^{k(2+\gamma_{ik}(x))}}{\mathcal{A}_{k}^{1+\gamma_{ik}(x)}}\right\}\mathcal{H}\left(\frac{\mathcal{A}_{k}}{\varrho^{k}}t,\varrho^{k}x\right),\\
		{f_{k}}(x):=&\max_{i=0,1,\ldots,N}\left\{\frac{\varrho^{k(2+\gamma_{ik}(x))}}{\mathcal{A}_{k}^{1+\gamma_{ik}(x)}}\right\}\|f\|_{L^\infty(B_1)}.
	\end{align*}
	Note that ${F_{k}}$ is a uniformly $(\lambda,\Lambda)$-elliptic operator and $0\leq \gamma_{1}\leq \gamma_{ik}(x)\leq \sigma_{ik}(x)\leq \gamma_{2}<\infty$, $i=0,1,\ldots,N$. A straightforward calculation yields 
	\begin{equation*}
		\|{\rm osc}_{{F_{k}}}\|_{L^{\infty}\left(B_{1}\right)}= 	\|{\rm osc}_{{F}}\|_{L^{\infty}\left(B_{\varrho^{k}}\right)}\leq \|{\rm osc}_{{F}}\|_{L^{\infty}\left(B_{1}\right)}\leq \iota.
	\end{equation*}
	By induction assumption, we have
	$$\|{u_{k}}\|_{L^{\infty}\left(B_{1}\right)}\leq 1.$$
	Applying \eqref{15} and the definition of $\mathcal{A}_{k}$, in combination of $\varrho\in(0,\frac{1}{2})$ and \eqref{11feidijianzhibiaoxulie}, we arrive at
	\begin{equation*}
		\|{f_{k}}\|_{L^{\infty}\left(B_{1}\right)}\leq \max_{i=0,1,\cdots,N}\left\{\varrho^{k(1-\alpha_{k}(1+\gamma_{ik}(x)))}\right\}\|{f}\|_{L^{\infty}\left(B_{1}\right)} \leq \|{f}\|_{L^{\infty}\left(B_{1}\right)},
	\end{equation*}
	\begin{equation*}
		\begin{split}
			|\mathcal{H}_{k}(t,x)|&\leq \varrho^{k}\max_{i=0,1,\cdots,N}\left\{\frac{\varrho^{k(1+\gamma_{ik}(x))}}{\mathcal{A}_{k}^{1+\gamma_{ik}(x)}}\right\}\left(\mathcal{K}+\mathcal{M}\left(\frac{\mathcal{A}_{k}}{\varrho^{k}}\right)^{m}|t|^{m}\right)\\
			&\leq \max_{i=0,1,\cdots,N}\left\{\varrho^{k(1-\alpha_{k}(1+\gamma_{ik}(x)))}\right\}\mathcal{K}+\mathcal{M}\varrho^{k(1-\alpha_{k}(1+\gamma_{ik}(x)-m))}|t|^{m}\\
			& \leq\mathcal{K}+\mathcal{M}|t|^{m}.
		\end{split}
	\end{equation*} 
	At this moment, the assumptions in Lemma \ref{lem5.21} are satisfied. Thus, we can apply Lemma \ref{lem5.21} to obtain
	\begin{equation*}
		\sup_{B_{\varrho}}\Abs{u_{k}(x)-u_{k}(0)}\leq \varrho^{1+\vartheta}+|Du_{k}(0)|\varrho.
	\end{equation*}
	Scaling back, we obtain
	\begin{equation}\label{09900}
		\begin{split}
			\sup_{B_{\varrho^{k+1}}}\Abs{u(x)-u(0)}&\leq \varrho^{1+\vartheta}\mathcal{A}_{k}+|Du(0)|\varrho^{k+1}\\
			&= \varrho^{k(1+\alpha_{k})+1+\vartheta}+|Du(0)|\left(\varrho^{1+k}+\varrho^{1+\vartheta}\sum_{i=0}^{k-1}\varrho^{k+i\alpha_{k}}\right)=:I+II.
		\end{split}
	\end{equation}
{\bf Estimate of $I$.} Applying \eqref{19} to obtain
\begin{equation*}
	\begin{split}
		k\left(\frac{1}{1+\gamma_{i}(0)}-\frac{1}{1+\max_{B_{\varrho^{k}}}\gamma_{i}(x)}\right)&\leq k\left(\max_{B_{\varrho^{k}}}\gamma_{i}(x)-\gamma_{i}(0)\right)\\
		&\leq k\omega(\varrho^{k}),
	\end{split}
\end{equation*}
which means that
\begin{equation}\label{09241}
	0\leq k(\alpha-\alpha_{k})\leq k\omega(\varrho^{k}).
\end{equation}
Since $\alpha_{k}\nearrow \alpha$, by virtue of Remark \ref{rmk1}, \eqref{09241}, and $\varrho<1$, we arrive at
\begin{equation*}
	\varrho^{k(\alpha_{k}-\alpha_{k+1})}=\varrho^{k(\alpha_{k}-\alpha)}\varrho^{k(\alpha-\alpha_{k+1})}\leq \varrho^{k(\alpha_{k}-\alpha)}\leq\varrho^{ k\omega(\varrho^{k})}\leq \varrho^{\frac{\alpha-\alpha_{0}}{2}}.
\end{equation*}
Set $\vartheta:=\frac{\alpha+\alpha_{0}}{2}<\alpha_{0}$. Then it follows that
\begin{equation}\label{09910}
	\begin{split}
		I&=\varrho^{(k+1)(1+\alpha_{k+1})}\varrho^{k(\alpha_{k}-\alpha_{k+1})}\varrho^{\vartheta-\alpha_{k+1}}\\
		&\leq \varrho^{(k+1)(1+\alpha_{k+1})}\varrho^{\frac{\alpha-\alpha_{0}}{2}}\varrho^{\vartheta-\alpha_{k+1}}\\ 
		&=\varrho^{(k+1)(1+\alpha_{k+1})}\varrho^{\alpha-\alpha_{k+1}}
		\\
		&\leq \varrho^{(k+1)(1+\alpha_{k+1})}.
	\end{split}
\end{equation}
{\bf Estimate of $II$.} Note that
\begin{equation}\label{09911}
	k(\alpha_{k+1}-\alpha_{k})=k(\alpha_{k+1}-\alpha)+k(\alpha-\alpha_{k})\leq k(\alpha_{k+1}-\alpha)+\frac{\alpha_{0}-\alpha}{2}.
\end{equation}
Then it follows from $\vartheta=\frac{\alpha+\alpha_{0}}{2}$ and $\alpha_{k}\nearrow \alpha$ that
\begin{equation}\label{09912}
	k(\alpha_{k+1}-\alpha)+\frac{\alpha_{0}-\alpha}{2}+\alpha_{k}-\vartheta=j(\alpha_{k+1}-\alpha)+\alpha_{k}-\alpha\leq 0.
\end{equation}
Combining \eqref{09911} with \eqref{09912} yields that
\begin{equation*}
	k(\alpha_{k+1}-\alpha_{k})+\alpha_{k}-\vartheta\leq 0.
\end{equation*}
This along with $\alpha_{k+1}-\alpha_{k}\geq 0$ yields that
\begin{equation*}
	j\alpha_{k+1}-(j-1)\alpha_{k}=j(\alpha_{k+1}-\alpha_{k})+\alpha_{k}\leq k(\alpha_{k+1}-\alpha_{k})+\alpha_{k}\leq  \vartheta,\quad j=1,2,\ldots,k,
\end{equation*}
that is,
\begin{equation*}
	j\alpha_{k}+\vartheta\geq (j+1)\alpha_{k+1},\quad j=0,1,\ldots,k-1.
\end{equation*}
Then it follows that
	\begin{equation}\label{09913}
		\begin{split}
			II\leq |Du(0)|\left(\varrho^{1+k}+\sum_{i=0}^{k-1}\varrho^{k+1+(i+1)\alpha_{k+1}}\right)=|Du(0)|\sum_{i=0}^{k}\varrho^{k+1+i\alpha_{k+1}}.
		\end{split}
	\end{equation}
	Substituting estimates \eqref{09910} and \eqref{09913} into \eqref{09900}, we obtain 
	\begin{equation*}
		\begin{split}
			\sup_{B_{\varrho^{k+1}}}\Abs{u(x)-u(0)}&\leq \varrho^{(k+1)(1+\alpha_{k+1})}+|Du(0)|\sum_{i=0}^{k}\varrho^{k+1+i\alpha_{k+1}}.
		\end{split}
	\end{equation*}
	This completes the proof of the desired result.
\end{proof}
\begin{lemma}\label{lem5.41} 
	Suppose that the assumptions of Lemma \ref{lem5.31} are in force. Then, there exists a universal constant $C_{0}>1$ such that
	\[
	\sup_{B_r} \frac{|u(x) - u(0)|}{r^{1+\alpha}} \leq C_0 \bigg(1 + |Du(0)|\, r^{-\alpha}\bigg), \quad \forall \,r \in (0, \varrho],
	\]
	where constant $\varrho$ comes from Lemma \ref{lem5.31}.
\end{lemma}
\begin{proof}
	Fix any small $r \in (0, \varrho]$, there exists $k\in\mathbb{N}$ such that $\varrho^{k+1}<r\leq \varrho^{k}$. By Lemma \ref{lem5.31}, we have
	\begin{equation*}
		\begin{split}
			\sup_{B_r} \frac{|u(x) - u(0)|}{r^{1+\alpha}}&\leq \frac{1}{\varrho^{1+\alpha}}\sup_{B_{\varrho^{k}}}\frac{\Abs{u(x)-u(0)}}{\varrho^{k(1+\alpha)}}\\
			&\leq \frac{1}{\varrho^{1+\alpha}}\left(\frac{\varrho^{k(1+\alpha_{k})}+|Du(0)|\sum_{i=0}^{k-1}\varrho^{k+i\alpha_{k}}}{\varrho^{k(1+\alpha)}}\right)\\
			&=\frac{1}{\varrho^{1+\alpha}}\left(\varrho^{k(\alpha_{k}-\alpha)}+|Du(0)|\varrho^{-k\alpha}\sum_{i=0}^{k-1}\varrho^{i\alpha_{k}}\right)\\
			&\leq \frac{1}{\varrho^{1+\alpha}}\left(\varrho^{k(\alpha_{k}-\alpha)}+|Du(0)|\varrho^{-k\alpha}\frac{1}{1-\varrho^{\alpha_{k}}}\right)\\
			&=\frac{\varrho^{k(\alpha_{k}-\alpha)}}{\varrho^{1+\alpha}}\left(1+|Du(0)|\varrho^{-k\alpha_{k}}\frac{1}{1-\varrho^{\alpha_{k}}}\right)\\
			&\leq \frac{\varrho^{k(\alpha_{k}-\alpha)}}{\varrho^{1+\alpha}}\left(1+|Du(0)|\varrho^{-k\alpha}\frac{1}{1-\varrho^{\alpha_{k}}}\right)\\
			&\leq \frac{\varrho^{k(\alpha_{k}-\alpha)}}{\varrho^{1+\alpha}}\left(1+|Du(0)|r^{-\alpha}\frac{1}{1-\varrho^{\alpha_{k}}}\right)\\
			&\leq C(\varrho)\left(1+|Du(0)|r^{-\alpha}\right),
		\end{split}
	\end{equation*}
	where we used the fact that $\limsup_{k\rightarrow \infty}k(\alpha_{k}-\alpha)=0$ and $\alpha_{k}\nearrow \alpha$. This completes the proof.
\end{proof}
\subsection{Proof of Theorem \ref{thm2}}
Finally, with the help of Lemma \ref{lem5.41}, we now complete the proof of Theorem \ref{thm2}.
\begin{proof}[Proof of Theorem \ref{thm2}] Initially, from Proposition \ref{prop4.1}, we can assume that
	\begin{equation}\label{2asmall}
		\|{u}\|_{L^{\infty}(B_{1})}\leq 1, \quad  \max\left\{\|{\rm osc}_{{F}}\|_{L^{\infty}\left(B_{1}\right)},\|{f}\|_{L^{\infty}(B_{1})},\mathcal{K},\mathcal{M}\right\}\leq \iota
	\end{equation} 
	for a constant $\iota>0$ coming from Lemma \ref{lem5.21}. In addition, it suffices to show $u\in C^{1,\alpha}(0)$. Fix any $r\in (0,\varrho]$, then  we analyze all the possible cases.\\
	{\bf Case 1.} If $|Du(0)|\leq r^{\alpha}$, by using Lemma \ref{lem5.41}, we obtain
	\begin{equation*}
		\begin{split}
			\sup_{B_r} |u(x) - u(0)-Du(0)\cdot x|&\leq \sup_{B_r} |u(x) - u(0)|+|Du(0)|r\\
			&\leq C_0 \bigg(1 + |Du(0)|\, r^{-\alpha}\bigg)r^{1+\alpha}+r^{1+\alpha}\\
			&\leq 3C_{0}r^{1+\alpha},
		\end{split}
	\end{equation*}
which means that $u$ is of class $C^{1,\alpha}$ at 0.\\
	{\bf Case 2.} If $r^{\alpha}<|Du(0)|\leq \varrho^{\alpha}$, we denote $r_{0}:=|Du(0)|^{1/\alpha}$ and define
	\begin{equation*}
		u_{r_{0}}(x):=\frac{u(r_{0}x)-u(0)}{r_{0}^{1+\alpha}}\quad {\rm for \;} x\in B_{1}.
	\end{equation*}
	Note that $r_{0}\leq \varrho$. By using Lemma \ref{lem5.41} again, we immediately obtain
	\begin{equation}\label{9565}
		\begin{split}
			\sup_{B_{1}}|u_{r_{0}}(x)|=\sup_{B_{r_{0}}} \frac{|u(x) - u(0)|}{r_{0}^{1+\alpha}}\leq C_{0}\left(1+|Du(0)|r_{0}^{-\alpha}\right)=2C_{0}.
		\end{split}
	\end{equation}
	It is easy to check that $u_{r_{0}}$ is a viscosity subsolution to  
	\begin{equation*}
		\min_{i=0,1,\ldots,N} \left\{ \bigg(|Du_{r_{0}}|^{\hat{\gamma}_{i}(x)}+a_{{r_{0}}}(x)|Du_{{r_{0}}}|^{\hat{\sigma}_{i}(x)} \bigg)F_{{r_{0}}}\left( D^2u_{{r_{0}}},x\right)\right\}+ \mathcal{H}_{{r_{0}}}({Du_{{r_{0}}}}, x) = \|f_{{r_{0}}}\|_{L^\infty(B_1)} \quad  \text{in} \quad B_{1}
	\end{equation*}
	and a viscosity supersolution to
	\begin{equation*}
		\max_{i=0,1,\ldots,N} \left\{ \bigg(|Du_{r_{0}}|^{\hat{\gamma}_{i}(x)}+a_{{r_{0}}}(x)|Du_{{r_{0}}}|^{\hat{\sigma}_{i}(x)} \bigg)F_{{r_{0}}}\left( D^2u_{{r_{0}}},x\right)\right\}+ \mathcal{H}_{{r_{0}}}({Du_{{r_{0}}}}, x) = \|f_{{r_{0}}}\|_{L^\infty(B_1)} \quad  \text{in} \quad B_{1}
	\end{equation*}
	with $u_{r_{0}}(0)=0$, $|Du_{r_{0}}(0)|=\frac{|Du(0)|}{r_{0}^{\alpha}}=1$, and $\|{u}_{r_{0}}\|_{L^{\infty}(B_{1})}\leq 2C_{0}$, where
	\begin{align*}
		{F_{r_{0}}}(X,x):=&r_{0}^{1-\alpha}F\left(r_{0}^{\alpha-1}X,r_{0}x\right),\\
		\hat{\gamma}_{i}(x):=&\gamma_{i}(r_{0}x),\quad \hat{\sigma}_{i}(x):=\sigma_{i}(r_{0}x), \quad
		{a_{r_{0}}}(x):=r_{0}^{-\alpha({\hat{\gamma}_{i}(x)-\hat{\sigma}_{i}(x)})}a(r_{0} x),\\
		\mathcal{H}_{r_{0}}(t,x):=&\max_{i=0,1,\ldots,N}\left\{r_{0}^{1-\alpha(1+\hat{\gamma}_{i}(x))}\right\}\mathcal{H}\left(r_{0}^{\alpha}t,r_{0}x\right),\\
		{f_{r_{0}}}(x):=&\max_{i=0,1,\ldots,N}\left\{r_{0}^{1-\alpha(1+\hat{\gamma}_{i}(x))}\right\}\|f\|_{L^\infty(B_1)}.
	\end{align*}
	Clearly, ${F_{r_{0}}}$ is a uniformly $(\lambda,\Lambda)$-elliptic operator and $0\leq \gamma_{1}\leq \hat{\gamma}_{i}(x)\leq \hat{\sigma}_{i}(x)\leq \gamma_{2}<\infty$, $i=0,1,\ldots,N$. A straightforward calculation yields that
	\begin{equation*}
		\|{\rm osc}_{{F_{r_{0}}}}\|_{L^{\infty}\left(B_{1}\right)}= 	\|{\rm osc}_{{F}}\|_{L^{\infty}\left(B_{r_{0}}\right)}\leq \|{\rm osc}_{{F}}\|_{L^{\infty}\left(B_{1}\right)}\leq\iota.
	\end{equation*}
Applying $r_{0}<1$ and $0\leq \gamma_{1}\leq \hat{\gamma}_{i}(x)\leq \gamma_{2}<\infty$, in combination with \eqref{15}, we arrive at
\begin{equation*}
	\|{f_{r_{0}}}\|_{L^{\infty}\left(B_{1}\right)}\leq r_{0}^{1-\alpha(1+\gamma_{2})}\|{f}\|_{L^{\infty}\left(B_{1}\right)},
\end{equation*}
	\begin{equation*}
		\begin{split}
			|\mathcal{H}_{r_{0}}(t,x)|\leq r_{0}^{1-\alpha(1+\gamma_{2})}\left(\mathcal{K}+\mathcal{M}r_{0}^{\alpha m}|t|^{m}\right) =:\mathcal{K}_{r_{0}}+\mathcal{M}_{r_{0}}|t|^{m}.
		\end{split}
	\end{equation*}
	At this point, we can invoke Corollary \ref{coro1} to obtain $u_{r_{0}}\in C_{\rm loc}^{1,\alpha^{\prime}}(B_{1})$ for some $\alpha^{\prime}\in(0,\alpha_{0})\cap\left(\frac{1}{1+\gamma_{2}}\right]$, and there exists a universal constant $C> 0$ such that for all $x\in B_{1/2}$,
	$$\Abs{Du_{r_{0}}(x)-Du_{r_{0}}(0)}\leq C|x|^{\alpha^{\prime}},$$
	this along with $|Du_{r_{0}}(0)|=1$ immediately yields
	$$ 1-C|x|^{\alpha^{\prime}}\leq \Abs{Du_{r_{0}}(x)}\leq 1+C|x|^{\alpha^{\prime}}.$$
	Then we may take a small universal radius $\varrho_{0}>0$, independent of $r_{0}$, such that
\begin{equation}\label{0081}
	c_{0}\leq |Du_{r_{0}}(x)| \leq c_{0}^{-1} \quad {\rm in}\;\, B_{\varrho_{0}}
\end{equation}
	with a fixed constant $c_{0}\in (0,1)$. 
	Applying \eqref{0081}, we derive
	$$\Abs{\frac{\|{f_{r_{0}}}\|_{L^{\infty}\left(B_{1}\right)} +\mathcal{H}_{{r_{0}}}(D{u_{r_{0}}}, x)}{\min_{i=0,1,\ldots,N} \left\{ \bigg(|Du_{r_{0}}|^{\hat{\gamma}_{i}(x)}+a_{{r_{0}}}(x)|Du_{{r_{0}}}|^{\hat{\sigma}_{i}(x)} \bigg)\right\}}}\leq \tilde{C}$$	
	Hence, we know that $u_{r_{0}}$ satisfies   
	\begin{equation*} 
		F_{{r_{0}}}\left( D^2u_{{r_{0}}},x\right)\leq \tilde{C} \quad \text{in} \quad  B_{\varrho_{0}}
	\end{equation*}
and 
\begin{equation*} 
	F_{{r_{0}}}\left( D^2u_{{r_{0}}},x\right)\geq -\tilde{C} \quad \text{in} \quad  B_{\varrho_{0}}
\end{equation*}
in the viscosity sense. By applying the classical result \cite{Caff1,Caffarelli1989}, we obtain $u_{r_{0}}\in C_{\rm loc}^{1,\alpha_{0}^{-}}(B_{\varrho_{0}})$ with the estimate
	\begin{equation*}
		\sup\limits_{x\in B_{\varrho_{1}}}\Abs{u_{r_{0}}(x)-Du_{r_{0}}(0)\cdot x}\leq C\varrho_{1}^{1+\alpha}
	\end{equation*}
	for every $0<\varrho_{1}\leq \frac{\varrho_{0}}{2}$ and $0<\alpha<\alpha_{0}$. Scaling back, it yields that
	\begin{equation}\label{883zuoquyujielunchengli}
		\sup\limits_{x\in B_{t}}\Abs{u(x)-u(0)-Du(0)\cdot x}\leq Ct^{1+\alpha}
	\end{equation}
	for every $0<t\leq \frac{\varrho_{0} r_{0}}{2}$. It remains to show that the claim \eqref{883zuoquyujielunchengli} also holds on interval $\left(\frac{\varrho_{0} r_{0}}{2},r_{0}\right)$. When $t\in\left(\frac{\varrho_{0} r_{0}}{2},r_{0}\right)$, we apply \eqref{9565} and $|Du(0)|=r_{0}^{\alpha}$ to arrive at
	\begin{equation*}
		\begin{split}
			\sup\limits_{x\in B_{t}}\Abs{u(x)-u(0)-Du(0)\cdot x}&\leq \sup\limits_{x\in B_{r_{0}}}\Abs{u(x)-u(0)-Du(0)\cdot x}\\
			&\leq \sup\limits_{x\in B_{r_{0}}}\Abs{u(x)-u(0)}+|Du(0)|r_{0}\\
			& \leq (2C_{0}+1)r_{0}^{1+\alpha}\\
			&=(2C_{0}+1) \left(\frac{2}{\varrho}\right)^{1+\alpha}\left(\frac{\varrho r_{0}}{2}\right)^{1+\alpha}\\
			&\leq (2C_{0}+1) \left(\frac{2}{\varrho}\right)^{1+\alpha}t^{1+\alpha}.
		\end{split}
	\end{equation*}
Thereby we have shown $u$ is $C^{1,\alpha}$ at 0.\\
	{\bf Case 3.} If $|Du(0)|>\varrho^{\alpha}$, we consider an auxiliary function
	$$\hat{u}(x):=\frac{\varrho^{\alpha}}{|Du(0)|}u(x)$$
	and then $|D\hat{u}(0)|=\varrho^{\alpha}$, which is back to Case 2. The proof is complete now.
\end{proof}
\begin{appendices}
\section{H\"{o}lder continuity}\label{sec:app}
In this section, we give the detailed proof of Propositions \ref{prop3.1} and \ref{prop3.2}. Indeed, Proposition \ref{prop3.1} is a plain consequence of Lemmas \ref{lem1} and \ref{lem2} below. First, we demonstrate that solutions are Lipschitz continuous if $|\xi|$ is large enough. The proof relies on the celebrated Crandall-Ishii-Lions Lemma (see \cite[Theorem 3.2]{Crandle1}, \cite[Proposition II.3]{Lions1}) .

\begin{lemma}\label{lem1}
	Assume \eqref{11}-\eqref{15} hold with $0<m\leq \gamma_{1}$.
	Let $u\in C(B_{1})$ be a viscosity subsolution to \eqref{cc11} and a viscosity supersolution to \eqref{cc12} with $\|u\|_{L^{\infty}(B_{1})}\leq 1$. Then there exists a constant $\mathcal{J}_{0}=\mathcal{J}_{0}(d,\lambda,\Lambda,m,\gamma_{1},\|f\|_{L^{\infty}(B_{1})},\mathcal{K},\mathcal{M})>1$ such that if $\abs{\xi}\geq \mathcal{J}_{0}$, then $u$ is locally Lipschitz continuous in $B_{1}$.
		In addition, there holds that
		$$|u(x)-u(y)|\leq C|x-y|$$
		for all $x,y\in B_{3/4}$, where $C>0$ is a universal constant.
\end{lemma}
\begin{proof}
	The proof is based on the standard doubling of variables argument. Let us fix $0<r<r_{1}<1$ and consider the quantity
	\begin{equation}\label{541}
		\mathcal{G}(x_{0}):=\sup\limits_{(x,y)\in B_{r_{1}}\times B_{r_{1}}}\left\{u(x)-u(y)-L_{1}w(\abs{x-y})-L_{2}\Big(\lvert x-x_{0}\rvert^{2}+\lvert y-x_{0} \rvert^{2} \Big)\right\}
	\end{equation}
	for each $x_{0}\in B_{r}$, where
	$$
	w(s)=
	\begin{cases}
		s- w_{0}s^{1+\beta}& \text{if} \;\; 0\leq s\leq s_{0}:=\left(\frac{1}{(1+\beta)w_{0}}\right)^{1/\beta}, \\
		w(s_{0}) &\text{if}\;\;  s>s_{0},
	\end{cases}
	$$
	with $\beta\in(0,1)$ and $w_{0}\in\left(0,\frac{1}{(1+\beta)2^{\beta}}\right)$. Observe that $s_{0}>2$,
	$$w(s)\geq 0,\quad 0\leq w^{\prime}(s)\leq 1,\quad w^{\prime\prime}(s)\leq 0,\quad \forall s\geq 0.$$
	
	Our goal is to prove that there exist constants $L_{1},L_{2}>1$ such that $\mathcal{G}\leq 0$ for all $x_{0}\in B_{r}$. We argue by contradiction by assuming that there exists $\hat{x}_{0}\in B_{r}$ so that $\mathcal{G}(\hat{x}_{0})>0$ for all $L_{1},L_{2}>1$.  Consider $\psi,\Psi:\overline{B}_{r_{1}}\times\overline{B}_{r_{1}}\rightarrow \mathbb{R}$, defined by
	\begin{equation*}
		\begin{cases}
			\psi(x,y):=L_{1}w(\abs{x-y})+L_{2}\Big(\lvert x-\hat{x}_{0}\rvert^{2}+\lvert y-\hat{x}_{0} \rvert^{2} \Big),\\
			\Psi\left(x,y\right):=u(x)-u(y)-\psi(x,y).
		\end{cases}
	\end{equation*}
	Let $\left(\hat{x},\hat{y}\right)\in \overline{B}_{r_{1}}\times\overline{B}_{r_{1}}$ be a maximum point for $\Psi(x,y)$, i.e., $\Psi\left(\hat{x},\hat{y}\right)>0$. Note that $\hat{x}\neq\hat{y}$; otherwise the maximum of $\Psi$ would be nonpositive. It follows from $\|u\|_{L^{\infty}(B_{1})}\leq 1$ that
	\begin{equation*}
		L_{1}w(\abs{\hat{x}-\hat{y}})+L_{2}\Big(\lvert \hat{x}-\hat{x}_{0}\rvert^{2}+\lvert \hat{y}-\hat{x}_{0} \rvert^{2} \Big)<u(\hat{x})-u(\hat{y})\leq 2\|u\|_{L^{\infty}(B_{1})}\leq 2.
	\end{equation*}
	This together with the triangle inequality yields that
	\begin{equation}\label{534}
		\abs{\hat{x}-\hat{y}}\leq \abs{\hat{x}-\hat{x}_{0}}+\abs{\hat{y}-\hat{x}_{0}}\leq \sqrt{2\left(\lvert \hat{x}-x_{0}\rvert^{2}+\lvert \hat{y}-x_{0} \rvert^{2} \right)}\leq \frac{2}{\sqrt{L_{2}}}.
	\end{equation}
	Choosing $L_{2}>\frac{8}{(r_{1}-r)^{2}}$ so that $\abs{\hat{x}-\hat{y}}<1$ and
	$$|\hat{x}|\leq \lvert \hat{x}-\hat{x}_{0}\rvert+|\hat{x}_{0}|\leq \sqrt\frac{2}{L_{2}}
	+r<\frac{r+r_{1}}{2},\quad |\hat{y}|\leq \lvert \hat{y}-\hat{x}_{0}\rvert+|\hat{x}_{0}|\leq \sqrt\frac{2}{L_{2}}+r<\frac{r+r_{1}}{2}.$$
	This means that	
	$\hat{x},\hat{y}$ belongs to the interior of $B_{(r+r_{1})/2}$.
	
	Next, we invoke the Crandall-Ishii-Lions lemma (see \cite[Theorem 3.2]{Crandle1}) to assure
	the existence of a limiting subjet $\left(\xi_{\hat{x}},X\right)$ of $u$ at $\hat{x}$ and a limiting superjet $\left(\xi_{\hat{y}},Y\right)$ of $u$ at $\hat{y}$, such that the matrices $X,Y\in S^{d}$ satisfy the matrix inequality
	\begin{equation}\label{434matrix}
		\left(
		\begin{array}{cc}
			X & 0 \\
			0 & -Y \\
		\end{array}
		\right)
		\leq  \left(
		\begin{array}{cc}
			B & -B \\
			-B & B \\
		\end{array}
		\right)+
		(2L_{2}+\epsilon)
		\left(
		\begin{array}{cc}
			I & 0 \\
			0 & I \\
		\end{array}
		\right)
	\end{equation}
	with $\epsilon\in(0,1)$, that only depends on the norm of $B$ and can be made sufficiently small. Here,
	\begin{equation*}
		\begin{split}
			\xi_{\hat{x}}&:=D_{x}\psi(\hat{x},\hat{y})= L_{1}w^{\prime}(\abs{\hat{x}-\hat{y}})\frac{\hat{x}-\hat{y}}{\abs{\hat{x}-\hat{y}}}+2L_{2}(\hat{x}-\hat{x}_{0}),\\
			\xi_{\hat{y}}&:=-D_{y}\psi(\hat{x},\hat{y})=L_{1}w^{\prime}(\abs{\hat{x}-\hat{y}})\frac{\hat{x}-\hat{y}}{\abs{\hat{x}-\hat{y}}}-2L_{2}(\hat{y}-\hat{x}_{0}),
		\end{split}
	\end{equation*}
	\begin{equation}\label{538}
		\begin{split}
			B:=L_{1}\left[\frac{w^{\prime}(\abs{\hat{x}-\hat{y}})}{\abs{\hat{x}-\hat{y}}}I+\left(w^{\prime\prime}(\abs{\hat{x}-\hat{y}})-\frac{w^{\prime}(\abs{\hat{x}-\hat{y}})}{\abs{\hat{x}-\hat{y}}}\right)\frac{(\hat{x}-\hat{y})\otimes(\hat{x}-\hat{y})}{\abs{\hat{x}-\hat{y}}^{2}}\right].
		\end{split}
	\end{equation}
Furthermore, we have the following viscosity inequalities
\begin{equation*}
	\min_{i=1,2} \left\{ |\xi_{\hat{x}}+\xi|^{\gamma_i} F(X, \hat{x})\right\}+ \mathcal{H}({\xi_{\hat{x}}+\xi},\hat{x}) \leq\|f\|_{L^\infty(B_1)},
\end{equation*}
\begin{equation*}
	\max_{i=1,2} \left\{ |\xi_{\hat{y}}+\xi|^{\gamma_i} F(Y, \hat{y})\right\}+ \mathcal{H}({\xi_{\hat{y}}+\xi},\hat{y}) \geq-\|f\|_{L^\infty(B_1)}.
\end{equation*}
	Then it follows that		
	\begin{equation*}
		F(Y, \hat{y})-F(X, \hat{x})\geq \frac{-\|f\|_{L^\infty(B_1)}-\mathcal{H}(\xi_{\hat{y}}+\xi, \hat{y})}{|\xi_{\hat{y}}+\xi|^{\gamma_i}}-\frac{\|f\|_{L^\infty(B_1)}-\mathcal{H}(\xi_{\hat{x}}+\xi, \hat{x})}{|\xi_{\hat{x}}+\xi|^{\gamma_j}}:=\mathcal{D}_{1},\quad i,j=1,2.
	\end{equation*}
	{\bf Estimate of $\mathcal{D}_{1}$}. We first choose $L_{1}>L_{2}$, then it follows from $\Abs{\hat{x}-\hat{x}_{0}}< \frac{1}{2}$ and $0\leq w^{\prime}(s)\leq 1$ for $s>0$ that
	\begin{equation*}
		\abs{\xi_{\hat{x}}}\leq L_{1}w^{\prime}(\abs{\hat{x}-\hat{y}})+2L_{2}\Abs{\hat{x}-\hat{x}_{0}}< 2L_{1}.
	\end{equation*}
	Now we take $\mathcal{J}_{0}:=3L_{1}$. Since $\abs{\xi}\geq \mathcal{J}_{0}$, we have
	\begin{align}\label{5xjie}
		\abs{\xi_{\hat{x}}+\xi}\geq\abs{\xi}- \abs{\xi_{\hat{x}}}\geq \mathcal{J}_{0}-2L_{1}=L_{1}>1.
	\end{align}
	In exactly the same way, we get
	\begin{equation}\label{5yjie}
		\abs{\xi_{\hat{y}}+\xi}\geq L_{1}>1.
	\end{equation}
	Applying condition \eqref{15}, in combination with \eqref{5xjie}, \eqref{5yjie}, and $0<m\leq \gamma_{1}$, we obtain
	\begin{equation*}
		\begin{split}
			|\mathcal{D}_{1}|\leq& \frac{\|f\|_{L^{\infty}(B_{1})}+\mathcal{K}+\mathcal{M}\abs{\xi_{\hat{y}}+\xi}^{m}}{\abs{\xi_{\hat{y}}+\xi}^{\gamma_{i}}}+\frac{\|f\|_{L^{\infty}(B_{1})}+\mathcal{K}+\mathcal{M}\abs{\xi_{\hat{x}}+\xi}^{m}}{\abs{\xi_{\hat{x}}+\xi}^{\gamma_{j}}}\\
			\leq&2\left(\|f\|_{L^{\infty}(B_{1})}+\mathcal{K}\right)+\mathcal{M}\left(\abs{\xi_{\hat{y}}+\xi}^{m-\gamma_{i}}+\abs{\xi_{\hat{x}}+\xi}^{m-\gamma_{j}}\right)\\
			\leq &2\left(\|f\|_{L^{\infty}(B_{1})}+\mathcal{K}+\mathcal{M}\right).
		\end{split}
	\end{equation*}
	Then we arrive at
	\begin{equation}\label{536}
		\begin{split}
			F(Y, \hat{y})-F(X, \hat{x})
			\geq-2\left(\|f\|_{L^{\infty}(B_{1})}+\mathcal{K}+\mathcal{M}\right).
		\end{split}
	\end{equation}
	
	On the other hand, we estimate the upper bound of $F(Y, \hat{y})-F(X, \hat{x})$. Combining the uniform ellipticity of operator $F$ at the point $\hat{x}$ with \eqref{12}, we deduce
	\begin{equation}\label{550}
		\begin{split}
			F(Y, \hat{y})-F(X, \hat{x})=&F(Y, \hat{y})-F(X, \hat{y})+F(X,\hat{y})-F(X,\hat{x})\\
			\leq& -\mathcal{P}_{\lambda,\Lambda}^{-}(X-Y)+C_{F}\|X\|\abs{\hat{x}-\hat{y}}^{\theta}.
		\end{split}
	\end{equation}
	{\bf Estimate of $\mathcal{P}_{\lambda,\Lambda}^{-}(X-Y)$}. We apply matrix inequality \eqref{434matrix} to vectors of the form $(z,z)\in\mathbb{R}^{2d}$ with $\abs{z}=1$, to obtain
	\begin{equation}\label{superasati}
		\left\langle(X-Y)z,z\right\rangle\leq \left(4L_{2}+2\epsilon\right)|z|^{2}.
	\end{equation}
	This means that all the eigenvalues of $X-Y$ are less than or equal to $4L_{2}+2\epsilon$. In addition, applying \eqref{434matrix} to the vector $(\nu,-\nu)\in\mathbb{R}^{2d}$ where $\nu:=\frac{\hat{x}-\hat{y}}{\abs{\hat{x}-\hat{y}}}$, 
	we get
	\begin{equation*}
		\begin{split}
			\left\langle(X-Y)\nu,\nu\right\rangle\leq 4\left\langle A\nu,\nu\right\rangle+\left(4L_{2}+2\epsilon\right)|\nu|^{2}=4L_{1}w^{\prime\prime}(\abs{\hat{x}-\hat{y}})+4L_{2}+2\epsilon.
		\end{split}
	\end{equation*}	
	This yields that
	at least one eigenvalue of $X-Y$ is less than $4L_{1}w^{\prime\prime}(\abs{\hat{x}-\hat{y}})+4L_{2}+2\epsilon$. We choose $L_{1}>\frac{4L_{2}+2}{4\beta(1+\beta)w_{0}}$, then it follows from $\abs{\hat{x}-\hat{y}}<1$ and $\beta<1$ that
	\begin{align*}
		4L_{1}w^{\prime\prime}(\abs{\hat{x}-\hat{y}})+4L_{2}+2\epsilon&=-4L_{1}\beta(\beta+1)w_{0}\abs{\hat{x}-\hat{y}}^{\beta-1}+4L_{2}+2\epsilon\\
		&<-4L_{1}\beta(\beta+1)w_{0}+4L_{2}+2<0.
	\end{align*}
	This means that at least one eigenvalue of
	$X-Y$ is negative. By the definition of Pucci extremal operator, we immediately get
	\begin{equation}\label{551}
		\mathcal{P}_{\lambda,\Lambda}^{-}(X-Y)\geq -\left(\lambda+\Lambda(d-1)\right)(4L_{2}+2\epsilon)+4L_{1}\lambda\beta(\beta+1)\omega_{0}\abs{\hat{x}-\hat{y}}^{\beta-1}.
	\end{equation}
	{\bf Estimate of $C_{F}\|X\|\abs{\hat{x}-\hat{y}}^{\theta}$}. Applying \eqref{matrix} to the vectors $(z,0)\in\mathbb{R}^{2d}$ with $\abs{z}=1$,
	we get
	\begin{equation}\label{661}
		\left\langle Xz,z\right\rangle\leq \left\langle Bz,z\right\rangle+2L_{2}+\epsilon.
	\end{equation}
It follows from the definition of the matrix $B$ in \eqref{538} and $0\leq w^{\prime}(s)\leq 1$ for $s\geq 0$ that 
\begin{equation}\label{662}
	B\leq L_{1}\frac{w^{\prime}(\abs{\hat{x}-\hat{y}})}{\abs{\hat{x}-\hat{y}}}I\leq L_{1}\abs{\hat{x}-\hat{y}}^{-1}I.
\end{equation}
A combination of \eqref{661} with \eqref{662} and $\Abs{\hat{x}-\hat{y}}^{\theta}<1$ yields that
	\begin{equation}\label{552}
		C_{F}\|X\|\Abs{\hat{x}-\hat{y}}^{\theta}\leq  C_{F}L_{1}\Abs{\hat{x}-\hat{y}}^{\theta-1}+C_{F}\left(2L_{2}+\epsilon\right).
	\end{equation}
Substituting estimates \eqref{551} and \eqref{552} into \eqref{550}, we obtain
	\begin{equation*}
		\begin{split}
			F(Y, \hat{y})-F(X, \hat{x})\leq  C_{5}+C_{F}L_{1}\Abs{\hat{x}-\hat{y}}^{\theta-1}-4\lambda L_{1}\beta(\beta+1)w_{0}\abs{\hat{x}-\hat{y}}^{\beta-1}, 
		\end{split}
	\end{equation*}	
	where $C_{5}:=(\lambda+(d-1)\Lambda)(4L_{2}+2)+C_{F}\left(2L_{2}+1\right)$.
	This together with \eqref{536} yields that
	\begin{equation*}
		-2\left(\|f\|_{L^{\infty}(B_{1})}+\mathcal{K}+\mathcal{M}\right)
		\leq C_{5}+
		L_{1}\abs{\hat{x}-\hat{y}}^{\beta-1}\left(C_{F}\abs{\hat{x}-\hat{y}}^{\theta-\beta}-4\lambda w_{0}\beta(1+\beta)\right).
	\end{equation*}
	Choose $L_{2}\geq 4\left(\frac{C_{F}}{2\lambda\omega_{0}\beta(1+\beta)}\right)^{2/(\theta-\beta)}$, we exploit \eqref{534} and $\beta<\theta$ to derive
	\begin{equation}\label{azuihou11}
		C_{F}\Abs{\hat{x}-\hat{y}}^{\theta-\beta}\leq C_{F}\left(\frac{2}{\sqrt{L_{2}}}\right)^{\theta-\beta}\leq 2\lambda w_{0}\beta(1+\beta).
	\end{equation}
	By means of \eqref{azuihou11} and $\abs{\hat{x}-\hat{y}}^{\beta-1}<1$, we immediately arrive at
	\begin{equation}\label{553}
		-2(\|f\|_{L^{\infty}(B_{1})}+\mathcal{K}+\mathcal{M})\leq  C_{5}-2L_{1}\lambda w_{0}\beta(1+\beta)\abs{\hat{x}-\hat{y}}^{\beta-1}<C_{1}-2L_{1}\lambda w_{0}\beta(1+\beta).
	\end{equation}
	Consequently, selecting $L_{1}\geq \frac{C_{5}+2(\|f\|_{L^{\infty}(B_{1})}+\mathcal{K}+\mathcal{M})}{2\lambda w_{0}\beta(1+\beta)}$ leads to a contradiction with \eqref{553}.
	
	As has been stated above, we verify that $\mathcal{G}(x_{0})\leq 0$ for each $x_{0}\in B_{r}$, which yields that
	$$|u(x_{0})-u(y_{0})|\leq L_{1}|x_{0}-y_{0}|+L_{2}|x_{0}-y_{0}|^{2}\leq C|x_{0}-y_{0}| $$
	for every $x_{0},y_{0}\in B_{r}$. This completes the proof.
\end{proof}
Next, we verify, in the complementary case, that the solutions are H\"{o}lder continuous by using \cite[Theorem 1.1]{Silvestre2016JEMS}.
\begin{lemma}\label{lem2}
Assume \eqref{11}-\eqref{15} hold with $0<m\leq \gamma_{1}$.
Let $u\in C(B_{1})$ be a viscosity subsolution to \eqref{cc11} and a viscosity supersolution to \eqref{cc12} with $\|u\|_{L^{\infty}(B_{1})}\leq 1$. If $\abs{\xi}\leq \mathcal{J}_{0}$  with $\mathcal{J}_{0}$ being the same as that in Lemma \ref{lem1}, then $u\in C_{\rm loc}^{\beta_{0}}(B_{1})$ for some $\beta_{0}\in(0,1)$. In addition, there holds that
	$$|u(x)-u(y)|\leq C|x-y|^{\beta_{0}}$$
	for all $x,y\in B_{3/4}$, where $C>0$ is a universal constant.
\end{lemma}
\begin{proof}
	 Let $\zeta\in\mathbb{R}^{d}$ such that $\abs{\zeta}\geq 3\mathcal{J}_{0}$. Then it follows from $\abs{\xi}< \mathcal{J}_{0}$ that $\abs{\zeta+\xi}\geq 2\mathcal{J}_{0}>1$. We claim that $u$ is a viscosity subsolution of 
	\begin{equation}\label{441}
			\mathcal{P}^{-}_{\lambda,\Lambda}(D^{2}u)-\left(\mathcal{K}+\mathcal{M}\right)-\|f\|_{L^{\infty}(B_{1})}= 0\quad\quad {\rm in} \;\; B_{1}\cap\left\{|Du|\geq 3\mathcal{J}_{0}\right\}
	\end{equation}	
	and a viscosity supersolution of 
	\begin{equation}\label{442}
			\mathcal{P}^{+}_{\lambda,\Lambda}(D^{2}u)+\left(\mathcal{K}+\mathcal{M}\right)+\|f\|_{L^{\infty}(B_{1})}= 0 \quad\quad {\rm in} \;\; B_{1}\cap\left\{|Du|\geq 3\mathcal{J}_{0}\right\}.
	\end{equation}	
	Indeed, let $\varphi\in C^{2}(B_{1})$ and assume that $u-\varphi$ attains a local maximum at $x_{0}\in B_{1}\cap  \left\{|Du|\geq 3\mathcal{J}_{0}\right\}$, then it follows from $\gamma_{i}>0$, $i=1,2$ and $\mathcal{J}_{0}>1$ that 
	\begin{equation}\label{554}
|D\varphi(x_{0})+\xi|\geq \mathcal{J}_{0}^{\gamma_{i}}>1.
\end{equation}	
 Since $u$ is a viscosity subsolution of \eqref{cc11}, we immediately have 
		\begin{equation}\label{555}
		\min_{i=1,2} \left\{ |D\varphi(x_{0})+\xi|^{\gamma_i} F \left( D^2\varphi(x_{0}),x_{0}\right)\right\}+ \mathcal{H}({D\varphi(x_{0})+\xi}, x_{0}) \leq \|f\|_{L^\infty(B_1)}.
	\end{equation}
If $F \left( D^2\varphi(x_{0}),x_{0}\right)\leq 0$, then 
$$\min_{i=1,2} \left\{ |D\varphi(x_{0})+\xi|^{\gamma_i} F \left( D^2\varphi(x_{0}),x_{0}\right)\right\}=\max_{i=1,2} \left\{ |D\varphi(x_{0})+\xi|^{\gamma_i} \right\}F \left( D^2\varphi(x_{0}),x_{0}\right)$$
Applying \eqref{15}, in combination with \eqref{554} and $m\leq \gamma_{1}\leq \gamma_{2}$, we arrive at
\begin{equation*}
\Abs{\frac{\|f\|_{L^\infty(B_1)}-\mathcal{H}({D\varphi(x_{0})+\xi}, x_{0})}{\max_{i=1,2} \left\{ |D\varphi(x_{0})+\xi|^{\gamma_i} \right\}}}\leq \|f\|_{L^\infty(B_1)}+\mathcal{K}+\mathcal{M}.
\end{equation*}
Combining with the all information above, we deduce
\begin{equation*}
	F \left( D^2\varphi(x_{0}),x_{0}\right)\leq   \|f\|_{L^\infty(B_1)}+\mathcal{K}+\mathcal{M}.
\end{equation*}
On the other hand, if $F \left( D^2\varphi(x_{0}),x_{0}\right)\geq 0$, applying \eqref{15} and \eqref{555}, in combination with \eqref{554} and $m\leq \gamma_{1}\leq \gamma_{2}$, we derive
\begin{equation*}
	F \left( D^2\varphi(x_{0}),x_{0}\right)\leq   \Abs{\frac{\|f\|_{L^\infty(B_1)}-\mathcal{H}({D\varphi(x_{0})+\xi}, x_{0})}{\min_{i=1,2} \left\{ |D\varphi(x_{0})+\xi|^{\gamma_i} \right\}}}\leq \|f\|_{L^\infty(B_1)}+\mathcal{K}+\mathcal{M}.
\end{equation*}
In either case, from uniform ellipticity and recalling that $F(0,\cdot)=0$, we obtain
\begin{equation*}
	\mathcal{P}^{-}_{\lambda,\Lambda}(D^{2}\varphi(x_{0}))\leq F \left( D^2\varphi(x_{0}),x_{0}\right)\leq   \|f\|_{L^\infty(B_1)}+\mathcal{K}+\mathcal{M},
\end{equation*}
As a consequence, $u$ is a viscosity subsolution to \eqref{441}.
In a similar way, we can prove that $u$ is a viscosity supersolution to \eqref{442}.

	At this point, we are in an exact position to apply \cite[Theorem 1.1]{Silvestre2016JEMS} 
	to know that $u$ is local H\"{o}lder continuous. This completes the proof.	
\end{proof}
\begin{proof}[Proof of Proposition \ref{prop3.4}] 
	Let $M_{0}:=\left(\frac{\kappa_{0}}{\mathcal{M}}\right)^{\frac{1}{m-\gamma_{1}}}$. It follows from $m>\gamma_{1}$ and \eqref{control} that
	\begin{equation}\label{bound}
		M_{0}\geq 1\quad {\rm and} \quad \abs{\xi}\leq M_{0}.
	\end{equation}
	The proof is similar to that of Proposition \ref{prop3.1}. Here we substitute $w(t)$ in \eqref{541} with $t^{\gamma}$ for $\gamma\in(0,1)$ to be fixed. As argued before, let us fix $0<r<r_{1}<1$, it suffices to show that there exist constants $L_{1},L_{2}> 1$ such that
	\begin{equation*}
		\mathcal{G}_{1}(x_{0}):=\sup\limits_{(x,y)\in B_{r_{1}}\times B_{r_{1}}}\left\{u(x)-u(y)-L_{1}\lvert x-y \rvert^{\gamma}-L_{2}\Big(\lvert x-x_{0}\rvert^{2}+\lvert y-x_{0} \rvert^{2} \Big)\right\}\leq 0
	\end{equation*}
	for each $x_{0}\in B_{r}$. We argue by contradiction by assuming that there exists $\hat{x}_{0}\in B_{r}$ so that $\mathcal{G}_{1}(\hat{x}_{0})>0$ for all $L_{1},L_{2}>1$.  Consider $\psi_{1},\Psi_{1}:\overline{B}_{r_{1}}\times\overline{B}_{r_{1}}\rightarrow \mathbb{R}$, defined by
	\begin{equation*}
		\begin{cases}
			\psi_{1}(x,y):=L_{1}\lvert x-y \rvert^{\gamma}+L_{2}\Big(\lvert x-x_{0}\rvert^{2}+\lvert y-x_{0} \rvert^{2} \Big),\\
			\Psi_{1}\left(x,y\right):=u(x)-u(y)-\psi_{1}(x,y).
		\end{cases}
	\end{equation*}
	Let $\left(\hat{x},\hat{y}\right)\in \overline{B}_{r_{1}}\times\overline{B}_{r_{1}}$ be a maximum point for $\Psi_{1}(x,y)$, i.e., $\Psi_{1}\left(\hat{x},\hat{y}\right)>0$. Note that $\hat{x}\neq\hat{y}$; otherwise the maximum of $\Psi_{1}$ would be nonpositive. By the same arguments as before, we can show $\abs{\hat{x}-\hat{y}}\leq \frac{2}{\sqrt{L_{2}}}$ and choose $L_{2}>\frac{8}{(r_{1}-r)^{2}}$ so that $\abs{\hat{x}-\hat{y}}<1$ and $\hat{x},\hat{y}\in B_{(r+r_{1})/2}$.
	
	We are in a position to apply the Crandall-Ishii-Lions lemma (see \cite[Theorem 3.2]{Crandle1}) to assure
	the existence of a limiting subjet $\left(\xi_{\hat{x}}^{\prime},X\right)$ of $u$ at $\hat{x}$ and a limiting superjet $\left(\xi_{\hat{y}}^{\prime},Y\right)$ of $u$ at $\hat{y}$, such that the matrices $X,Y\in S^{d}$ satisfy the matrix inequality
	\begin{equation}\label{matrix}
		\left(
		\begin{array}{cc}
			X & 0 \\
			0 & -Y \\
		\end{array}
		\right)
		\leq  \left(
		\begin{array}{cc}
			B^{\prime} & -B^{\prime} \\
			-B^{\prime} & B^{\prime} \\
		\end{array}
		\right)+
		(2L_{2}+\epsilon)
		\left(
		\begin{array}{cc}
			I & 0 \\
			0 & I \\
		\end{array}
		\right)
	\end{equation}
	with $\epsilon\in(0,1)$, that only depends on the norm of $B^{\prime}$ and can be made sufficiently small. Here,
	\begin{equation*}
		\xi_{\hat{x}}^{\prime}:=\gamma L_{1}(\hat{x}-\hat{y})\lvert \hat{x}-\hat{y}\rvert^{\gamma -2}+2L_{2}(\hat{x}-x_{0}),  \quad \xi_{\hat{y}}^{\prime}:=\gamma L_{1}(\hat{x}-\hat{y})\lvert \hat{x}-\hat{y}\rvert^{\gamma -2}-2L_{2}(\hat{y}-x_{0}),
	\end{equation*}
	\begin{equation}\label{Adingyi}
		B^{\prime}:=L_{1}\gamma\left[(\gamma-2)\abs{\hat{x}-\hat{y}}^{\gamma-4}\left((\hat{x}-\hat{y})\otimes(\hat{x}-\hat{y})\right)+\abs{\hat{x}-\hat{y}}^{\gamma-2}I\right].
	\end{equation}
Furthermore, we have the following viscosity inequalities
\begin{equation*}
	\min_{i=1,2} \left\{ |\xi_{\hat{x}}^{\prime}+\xi|^{\gamma_i} F(X, \hat{x})\right\}+ \mathcal{H}({\xi_{\hat{x}}^{\prime}+\xi},\hat{x}) \leq\|f\|_{L^\infty(B_1)},
\end{equation*}
\begin{equation*}
	\max_{i=1,2} \left\{ |\xi_{\hat{y}}^{\prime}+\xi|^{\gamma_i} F(Y, \hat{y})\right\}+ \mathcal{H}({\xi_{\hat{y}}^{\prime}+\xi},\hat{y}) \geq-\|f\|_{L^\infty(B_1)},
\end{equation*}
Then it follows that		
\begin{equation*}
	F(Y, \hat{y})-F(X, \hat{x})\geq \frac{-\|f\|_{L^\infty(B_1)}-\mathcal{H}(\xi_{\hat{y}}^{\prime}+\xi, \hat{y})}{|\xi_{\hat{y}}^{\prime}+\xi|^{\gamma_i}}-\frac{\|f\|_{L^\infty(B_1)}-\mathcal{H}(\xi_{\hat{x}}^{\prime}+\xi, \hat{x})}{|\xi_{\hat{x}}^{\prime}+\xi|^{\gamma_j}}:=\mathcal{D}_{1}^{\prime},\quad i,j=1,2.
\end{equation*}
We first estimate $\mathcal{D}_{1}^{\prime}$. Choose $L_{1}>\frac{L_{2}2^{\gamma}}{\gamma(r_{1}-r)^{\gamma-2}}$, it follows from $\Abs{\hat{x}-\hat{y}}\leq\frac{r_{1}-r}{2}$ and $\gamma<1$ that
$$2L_{2}\Abs{\hat{x}-x_{0}}\leq 2L_{2}\frac{r_{1}-r}{2}< \frac{\gamma L_{1}}{2}\left(\frac{r_{1}-r}{2}\right)^{\gamma-1}\leq \frac{\gamma L_{1}}{2}\Abs{\hat{x}-\hat{y}}^{\gamma-1}.
$$	
This along with the triangle inequality leads to that
\begin{equation}\label{0921}
	\abs{\xi_{\hat{x}}^{\prime}}\geq \gamma L_{1}\lvert \hat{x}-\hat{y}\rvert^{\gamma -1}-2L_{2}\Abs{\hat{x}-x_{0}}\geq 
	\frac{\gamma L_{1}}{2}\Abs{\hat{x}-\hat{y}}^{\gamma-1},
\end{equation}
\begin{equation}\label{0922}
	\abs{\xi_{\hat{x}}^{\prime}}\leq \gamma L_{1}\lvert \hat{x}-\hat{y}\rvert^{\gamma -1}+2L_{2}\Abs{\hat{x}-x_{0}}
	\leq 2\gamma L_{1}\Abs{\hat{x}-\hat{y}}^{\gamma-1}.
\end{equation}
In exactly the same way, we derive
\begin{equation}\label{0923}
	\frac{\gamma L_{1}}{2}\Abs{\hat{x}-\hat{y}}^{\gamma-1}\leq \abs{\xi_{\hat{y}}^{\prime}}\leq 2\gamma L_{1}\Abs{\hat{x}-\hat{y}}^{\gamma-1}.
\end{equation}	
	We proceed to select $L_{1}>\frac{4M_{0}}{\gamma}$, then a combination of \eqref{bound} with \eqref{0921}-\eqref{0923} and $\Abs{\hat{x}-\hat{y}}^{\gamma-1}>1$ yields that
\begin{equation}\label{81}
	\begin{split}
		|\xi_{\hat{x}}^{\prime}+\xi|&\leq |\xi_{\hat{x}}^{\prime}| + |\xi| \leq 2\gamma L_1 |\hat{x}-\hat{y}|^{\gamma-1} + M_0 \leq 2\gamma L_1 |\hat{x}-\hat{y}|^{\gamma-1} + \frac{\gamma L_1}{4} \\
		&\leq 2\gamma L_1 |\hat{x}-\hat{y}|^{\gamma-1} + \gamma L_1 |\hat{x}-\hat{y}|^{\gamma-1} = 3\gamma L_1 |\hat{x}-\hat{y}|^{\gamma-1},
	\end{split} 
\end{equation}
\begin{equation}\label{82}
	\begin{split}
		|\xi_{\hat{x}}^{\prime}+\xi| &\geq |\xi_{\hat{x}}^{\prime}| - |\xi| \geq \frac{\gamma L_1}{2} |\hat{x}-\hat{y}|^{\gamma-1} - M_0 \geq \frac{\gamma L_1}{2} |\hat{x}-\hat{y}|^{\gamma-1} - \frac{\gamma L_1}{4} |\hat{x}-\hat{y}|^{\gamma-1} \\
		&= \frac{\gamma L_1}{4} |\hat{x}-\hat{y}|^{\gamma-1} > M_0 \, |\hat{x}-\hat{y}|^{\gamma-1} > 1.
	\end{split}
\end{equation}
By the same way, we can deduce
\begin{equation}\label{83}
	1<M_{0}\Abs{\hat{x}-\hat{y}}^{\gamma-1}\leq |\xi_{\hat{y}}^{\prime}+\xi|\leq 3\gamma L_{1}\Abs{\hat{x}-\hat{y}}^{\gamma-1}.
\end{equation}
	Applying \eqref{15}, in combination with
	\eqref{81}-\eqref{83}, $\gamma_{1}\leq \gamma_{2}$ and $\gamma_{1}<m\leq 1+\gamma_{1}$, we arrive at
	\begin{equation*}
		\begin{split}
			\abs{\mathcal{D}_{1}^{\prime}}&\leq \frac{\|f\|_{L^{\infty}(B_{1})}+\mathcal{K}+\mathcal{M}|\xi_{\hat{y}}^{\prime}+\xi|^{m}}{|\xi_{\hat{y}}^{\prime}+\xi|^{\gamma_{i}}}+\frac{\|f\|_{L^{\infty}(B_{1})}+\mathcal{K}+\mathcal{M}|\xi_{\hat{x}}^{\prime}+\xi|^{m}}{|\xi_{\hat{x}}^{\prime}+\xi|^{\gamma_{j}}}\\
			&\leq 2\|f\|_{L^{\infty}(B_{1})}+2\mathcal{K}+
			2\mathcal{M}\left(3\gamma L_{1}\Abs{\hat{x}-\hat{y}}^{\gamma-1}\right)^{m-\gamma_{1}}\\
			&\leq  2\left(\|f\|_{L^{\infty}(B_{1})}+\mathcal{K}\right)+6\mathcal{M}L_{1}\Abs{\hat{x}-\hat{y}}^{(\gamma-1)(m-\gamma_{1})}.
		\end{split}
	\end{equation*}
	Then it follows that
	\begin{equation}\label{991}
	 F(Y, \hat{y})-F(X, \hat{x})\geq -2\left(\|f\|_{L^{\infty}(B_{1})}+\mathcal{K}\right)-6\mathcal{M}L_{1}\Abs{\hat{x}-\hat{y}}^{(\gamma-1)(m-\gamma_{1})}.
	\end{equation}
	
	On the other hand, we estimate the upper bound of $F(Y, \hat{y})-F(X, \hat{x})$. For vectors of the form $(z,z)\in\mathbb{R}^{2d}$ with $\abs{z}=1$, we apply the matrix inequality \eqref{matrix} to obtain
	\begin{equation*}
		\left\langle(X-Y)z,z\right\rangle\leq \left(4L_{2}+2\epsilon\right)|z|^{2}.
	\end{equation*}
	This means that all the eigenvalues of $X-Y$ are less than or equal to $4L_{2}+2\epsilon$. In addition, applying \eqref{matrix} to the vector $(\nu,-\nu)\in\mathbb{R}^{2d}$ where $\nu:=\frac{\hat{x}-\hat{y}}{\abs{\hat{x}-\hat{y}}}$, 
	we get
	\begin{equation*}
		\begin{split}
			\left\langle(X-Y)\nu,\nu\right\rangle\leq 4\left\langle B^{\prime}\nu,\nu\right\rangle+\left(4L_{2}+2\epsilon\right)|\nu|^{2}=4L_{1}\gamma(\gamma-1)\abs{\hat{x}-\hat{y}}^{\gamma-2}+4L_{2}+2\epsilon.
		\end{split}
	\end{equation*}	
	This yields that
	at least one eigenvalue of $X-Y$ is less than $4L_{1}\gamma(\gamma-1)\abs{\hat{x}-\hat{y}}^{\gamma-2}+4L_{2}+2\epsilon$. We choose $L_{1}>\frac{4L_{2}+2}{4\gamma(1-\gamma)}$, then it leads to
	$$4L_{1}\gamma(\gamma-1)\abs{\hat{x}-\hat{y}}^{\gamma-2}+4L_{2}+2\epsilon<4L_{1}\gamma(\gamma-1)+4L_{2}+2<0.$$
	This means that at least one eigenvalue of
	$X-Y$ is negative. Combining the all information above with the definition of Pucci extremal operator, we arrive at
	\begin{equation}\label{221}
		\mathcal{P}_{\lambda,\Lambda}^{-}(X-Y)\geq -\left(\lambda+\Lambda(d-1)\right)(4L_{2}+2)-4L_{1}\gamma(\gamma-1)\abs{\hat{x}-\hat{y}}^{\gamma-2}.
	\end{equation}  
	Further, applying \eqref{matrix} to the vectors $(z,0)\in\mathbb{R}^{2d}$ with $\abs{z}=1$,
	we obtain
	\begin{equation}\label{855}
		\left\langle Xz,z\right\rangle\leq \left\langle B^{\prime}z,z\right\rangle+(2L_{2}+\epsilon)|z|^{2}.
	\end{equation}
	It follows from the definition of the matrix $B^{\prime}$ in \eqref{Adingyi} that $B^{\prime}\leq L_{1}\gamma\abs{\hat{x}-\hat{y}}^{\gamma-2}I$. This together with \eqref{855} yields that
\begin{equation}\label{222}
	C_{F}\|X\|\Abs{\hat{x}-\hat{y}}^{\theta}\leq  C_{F}L_{1}\Abs{\hat{x}-\hat{y}}^{\theta+\gamma-2}+C_{F}\left(2L_{2}+1\right).
\end{equation}
Combining \eqref{221} with \eqref{222}, we obtain
\begin{equation*}
	\begin{split}
		F(Y, \hat{y})-F(X, \hat{x})\leq  C_{5}+C_{F}\gamma L_{1}\Abs{\hat{x}-\hat{y}}^{\theta+\gamma-2}+4L_{1}\gamma(\gamma-1)\abs{\hat{x}-\hat{y}}^{\gamma-2}, 
	\end{split}
\end{equation*}	
where $C_{5}:=(\lambda+(d-1)\Lambda)(4L_{2}+2)+C_{F}\left(2L_{2}+1\right)$.
This together with \eqref{991} yields that
\begin{equation*}
	\begin{split}
		&-2\left(\|f\|_{L^{\infty}(B_{1})}+\mathcal{K}\right)-C_{5}\\
		\leq& L_{1}\Abs{\hat{x}-\hat{y}}^{\gamma-2}\left(C_{F}\gamma\Abs{\hat{x}-\hat{y}}^{\theta}-4\gamma(1-\gamma)+6\mathcal{M}\Abs{\hat{x}-\hat{y}}^{(\gamma-1)(m-\gamma_{1}-1)+1}\right).
	\end{split}
\end{equation*}
Choose 
\begin{equation}\label{993}
L_{2}\geq \max\left\{4\left(\frac{C_{F}}{\lambda(1-\gamma)}\right)^{2/\theta},\left(\frac{12\mathcal{M}}{\lambda\gamma(1-\gamma)}\right)^{2}\right\}.
\end{equation}
Applying $\abs{\hat{x}-\hat{y}}\leq \frac{2}{\sqrt{L_{2}}}$ and \eqref{993} to derive
\begin{equation*}
	C_{F}\Abs{\hat{x}-\hat{y}}^{\theta}\leq C_{F}\left(\frac{2}{\sqrt{L_{2}}}\right)^{\theta}\leq \lambda(1-\gamma),\quad 6\mathcal{M}\Abs{\hat{x}-\hat{y}}\leq \frac{12\mathcal{M}}{\sqrt{L_{2}}}\leq \lambda\gamma(1-\gamma).
\end{equation*}
Gathering the previous estimates with $\abs{\hat{x}-\hat{y}}^{\gamma-2}>1$, we obtain
\begin{equation}\label{key11}
	-2\left(\|f\|_{L^{\infty}(B_{1})}+\mathcal{K}\right)
	\leq C_{5}-2L_{1}\lambda\gamma(1-\gamma)\abs{\hat{x}-\hat{y}}^{\gamma-2}< C_{5}-2L_{1}\lambda\gamma(1-\gamma).
\end{equation}
	Finally, choosing $L_{1}\geq \frac{C_{5}+2\left(\|f\|_{L^{\infty}(B_{1})}+\mathcal{K}\right)}{2\lambda\gamma(1-\gamma)}$ leads to a contradiction with \eqref{key11}. The proof is complete.	
\end{proof}
\end{appendices}
\section*{Data availability} Data sharing is not applicable to this article as obviously no datasets were generated or analyzed during the current study.
\section*{Conflict of interest} Author states no conflict of interest.

\end{document}